\documentclass[11pt]{amsart}

\usepackage{amsmath}
\usepackage{amssymb}
\usepackage{amscd}

\usepackage{xypic}
\xyoption{all}

\newcommand{\nc}{\newcommand}

\nc{\exto}[1]{\stackrel{#1}{\longrightarrow}}
\nc{\dlim}{{\mathop{\lim\limits_{\longrightarrow}\,}}}
\nc{\ilim}{{\mathop{\lim\limits_{\longleftarrow}\,}}}
\nc{\hocolim}{{\mathop{\sf hocolim}\,}}
\nc{\holim}{{\mathop{\sf holim}}}
\nc{\lan}{\big\langle}
\nc{\ran}{\big\rangle}

\nc{\kk}{{\mathsf{k}}}

\nc{\C}{{\mathbb{C}}}
\nc{\HH}{{\mathbf{H}}}
\nc{\DD}{{\mathbb{D}}}
\nc{\LL}{{\mathbb{L}}}
\nc{\PP}{{\mathbb{P}}}
\nc{\QQ}{{\mathbb{Q}}}
\nc{\RR}{{\mathbb{R}}}
\nc{\ZZ}{{\mathbb{Z}}}

\nc{\CA}{{\mathcal{A}}}
\nc{\CB}{{\mathcal{B}}}
\nc{\CC}{{\mathcal{C}}}
\nc{\D}{{\mathcal{D}}}
\nc{\CE}{{\mathcal{E}}}
\nc{\CF}{{\mathcal{F}}}
\nc{\CG}{{\mathcal{G}}}
\nc{\CH}{{\mathcal{H}}}
\nc{\CL}{{\mathcal{L}}}
\nc{\CM}{{\mathcal{M}}}
\nc{\CN}{{\mathcal{N}}}
\nc{\CO}{{\mathcal{O}}}
\nc{\CQ}{{\mathcal{Q}}}
\nc{\CR}{{\mathcal{R}}}
\nc{\CS}{{\mathcal{S}}}
\nc{\CT}{{\mathcal{T}}}
\nc{\CU}{{\mathcal{U}}}
\nc{\CV}{{\mathcal{V}}}
\nc{\CW}{{\mathcal{W}}}
\nc{\CX}{{\mathcal{X}}}
\nc{\CY}{{\mathcal{Y}}}
\nc{\CMo}{{\mathcal{M}^\circ}}
\nc{\Co}{{{C}^\circ}}

\nc{\BY}{{\overline{Y}}}
\nc{\BYD}{{\overline{Y}{}^{|D|}}}
\nc{\OZ}{{\overline{Z}}}
\nc{\bg}{{\bar{g}}}

\nc{\bq}{{\mathbf{q}}}
\nc{\BD}{{\mathbf{D}}}
\nc{\BG}{{\mathbf{G}}}
\nc{\BM}{{\mathbf{M}}}
\nc{\BP}{{\mathbf{P}}}
\nc{\BZ}{{\mathbf{Z}}}
\nc{\BPr}{{\mathsf{P}}}
\nc{\BL}{{\mathbf{L}}}
\nc{\BR}{{\mathbf{R}}}
\nc{\BRO}[1]{{{\mathbf{R}}^{\circ}_{#1}}}
\nc{\BRD}[1]{{{\mathbf{R}}^{|D|}_{#1}}}
\nc{\BRP}[1]{{{\mathbf{R}}^{1}_{#1}}}
\nc{\BRTP}[1]{{{\mathbf{\tilde{R}}}{}^{1}_{#1}}}
\nc{\BS}{{\mathbf{S}}}
\nc{\BMS}{{{\mathbf{M}}^{{s}}}}
\nc{\BMSS}{{{\mathbf{M}}^{{ss}}}}
\nc{\BMZ}{{\mathbf{M}^{\circ}}}
\nc{\BCL}{{\mathbf{L}}}

\nc{\PCC}{{{}^\perp\CC}}

\nc{\Cl}{{\mathsf{Cliff}}}
\nc{\Clev}{{\mathop{\mathsf{Cliff}}^{\circ}}}

\nc{\FA}{{\mathfrak{A}}}
\nc{\FB}{{\mathfrak{B}}}

\nc{\fa}{{\mathfrak{a}}}
\nc{\fb}{{\mathfrak{b}}}
\nc{\fg}{{\mathfrak{g}}}
\nc{\fn}{{\mathfrak{n}}}
\nc{\fp}{{\mathfrak{p}}}
\nc{\FD}{{\mathfrak{D}}}
\nc{\FE}{{\mathfrak{E}}}
\nc{\FL}{{\mathfrak{L}}}
\nc{\FM}{{\mathfrak{M}}}
\nc{\FS}{{\mathsf{S}}}

\nc{\sfc}{{\mathsf{c}}}
\nc{\sfch}{{\mathsf{ch}}}
\nc{\sfh}{{\mathsf{h}}}

\nc{\SK}{{\mathsf{K}}}
\nc{\SM}{{\mathsf{M}}}
\nc{\SO}{{\mathsf{O}}}
\nc{\SQ}{{\mathsf{Q}}}
\nc{\SPV}{{\mathsf{S}^+\mathsf{V}}}
\nc{\SMV}{{\mathsf{S}^-\mathsf{V}}}
\nc{\SPMV}{{\mathsf{S}^\pm\mathsf{V}}}
\nc{\SX}{{S_X}}
\nc{\SY}{{S_Y}}
\nc{\phipsi}{{q}}
\nc{\eps}{\varepsilon}

\nc{\pim}{{\pi_-}}
\nc{\pip}{{\pi_+}}

\nc{\BE}{{\overline{\CE}}}
\nc{\TE}{{\tilde{\CE}}}
\nc{\TQ}{{\tilde{Q}}}
\nc{\TCF}{{\tilde{\CF}}}
\nc{\TCG}{{\tilde{\CG}}}
\nc{\TCL}{{\tilde{\CL}}}
\nc{\TF}{{\tilde{F}}}
\nc{\TW}{{\tilde{W}}}
\nc{\TCC}{{\tilde{\CC}}}
\nc{\TCX}{{\tilde{\CX}}}
\nc{\TCY}{{\tilde{\CY}}}
\nc{\TPhi}{{\tilde{\Phi}}}
\nc{\OPhi}{{\bar{\Phi}}}
\nc{\txi}{{\tilde{\xi}}}
\nc{\tp}{{\tilde{p}}}
\nc{\tq}{{\tilde{q}}}
\nc{\tzeta}{{\tilde{\zeta}}}
\nc{\tpi}{{\tilde{\pi}}}

\nc{\halpha}{{\hat{\alpha}}}
\nc{\HCA}{{\hat{\CA}}}
\nc{\HCB}{{\hat{\CB}}}
\nc{\HCC}{{\hat{\CC}}}
\nc{\HE}{{\widehat{\CE}}}
\nc{\HX}{{\hat{X}}}
\nc{\hxi}{{\hat{\xi}}}

\nc{\UH}{{\mathcal{H}}}

\nc{\TM}{{\widetilde{M}}}
\nc{\TCM}{{\widetilde{\CM}}}
\nc{\TU}{{\widetilde{U}}}
\nc{\TX}{{\widetilde{X}}}
\nc{\TY}{{\widetilde{Y}}}
\nc{\TYO}{{{\widetilde{Y}}^\circ}}
\nc{\barf}{{\bar{f}}}
\nc{\te}{{\tilde{e}}{}}
\nc{\tf}{{\tilde{f}}}
\nc{\tg}{{\tilde{g}}}
\nc{\ti}{{\tilde{\imath}}}
\nc{\tj}{{\tilde{\jmath}}}
\nc{\ty}{{\tilde{y}}}
\nc{\tphi}{{\tilde{\phi}}}

\nc{\urho}{{\underline{\rho}}}

\nc{\LRA}{\Leftrightarrow}
\nc{\RA}{\Rightarrow}
\nc{\lotimes}{\mathbin{\mathop{\otimes}\limits^{\mathbb{L}}}}
\nc{\CEnd}{\mathop{\mathcal{E}\mathit{nd}}\nolimits}
\nc{\CExt}{\mathop{\mathcal{E}\mathit{xt}}\nolimits}
\nc{\CHom}{\mathop{\mathcal{H}\mathit{om}}\nolimits}
\nc{\RH}{\mathop{{\mathsf{R}}\Gamma}\nolimits}
\nc{\RGamma}{\mathop{{\mathsf{R}}\Gamma}\nolimits}
\nc{\RHom}{\mathop{\mathsf{RHom}}\nolimits}
\nc{\RCHom}{\mathop{\mathsf{R}\mathcal{H}\mathit{om}}\nolimits}
\nc{\RG}{\mathop{\mathsf{R\Gamma}}\nolimits}
\nc{\Hom}{\mathop{\mathsf{Hom}}\nolimits}
\nc{\Ext}{\mathop{\mathsf{Ext}}\nolimits}
\nc{\End}{\mathop{\mathsf{End}}\nolimits}
\nc{\Tor}{\mathop{\mathsf{Tor}}\nolimits}
\nc{\Tordim}{\mathop{\mathsf{Tor}\text{\rm-}\mathsf{dim}}\nolimits}
\nc{\Hilb}{\mathop{\mathsf{Hilb}}\nolimits}
\nc{\Spec}{\mathop{\mathsf{Spec}}\nolimits}
\nc{\Pic}{\mathop{\mathsf{Pic}}\nolimits}
\renewcommand{\Im}{\mathop{\mathsf{Im}}\nolimits}
\nc{\Tr}{\mathop{\mathsf{Tr}}\nolimits}
\nc{\Cone}{\mathop{\mathsf{Cone}}\nolimits}
\nc{\Fiber}{\mathop{\mathsf{Fiber}}\nolimits}
\nc{\Ker}{\mathop{\mathsf{Ker}}\nolimits}
\nc{\Coker}{\mathop{\mathsf{Coker}}\nolimits}
\nc{\codim}{\mathop{\mathsf{codim}}\nolimits}
\nc{\sing}{{\mathsf{sing}}}
\nc{\supp}{\mathop{\mathsf{supp}}}
\nc{\vol}{\mathop{\mathsf{vol}}\nolimits}
\nc{\ch}{\mathop{\mathsf{ch}}\nolimits}
\nc{\perf}{{\mathsf{perf}}}
\nc{\rank}{\mathop{\mathsf{rank}}}
\nc{\Pf}{{\mathsf{Pf}}}
\nc{\Gr}{{\mathsf{Gr}}}
\nc{\OGr}{{\mathsf{OGr}}}
\nc{\Flag}{{\mathsf{Fl}}}
\nc{\Kosz}{{\mathsf{Kosz}}}
\nc{\LGr}{{\mathsf{LGr}}}
\nc{\GTGr}{{\mathsf{G_2Gr}}}
\nc{\GTF}{{\mathsf{G_2F}}}
\nc{\OF}{{\mathsf{OF}}}
\nc{\Fl}{{\mathsf{Fl}}}
\nc{\Bl}{{\mathsf{Bl}}}
\nc{\GL}{{\mathsf{GL}}}
\nc{\PGL}{{\mathsf{PGL}}}
\nc{\SL}{{\mathsf{SL}}}
\nc{\SP}{{\mathsf{Sp}}}
\nc{\Spin}{{\mathsf{Spin}}}
\nc{\Tot}{{\mathsf{Tot}}}
\nc{\ev}{{\mathsf{ev}}}
\nc{\od}{{\mathsf{odd}}}
\nc{\coev}{{\mathsf{coev}}}
\nc{\id}{{\mathsf{id}}}
\nc{\opp}{{\mathsf{opp}}}
\nc{\PS}{{{\PP^3}}}
\nc{\Qu}{{{Q^3}}}
\nc{\tdim}{\mathop{\Tor\dim}}
\nc{\ecart}{{\fbox{$\scriptstyle\mathsf{EC}$}}}
\nc{\ad}{{\mathop{\mathsf ad}}}
\nc{\sg}{{\mathop{\mathsf sg}}}
\nc{\hf}{{\mathop{\mathsf hf}}}
\nc{\gr}{{\mathop{\mathsf gr}}}
\nc{\qgr}{{\mathop{\mathsf qgr}}}
\nc{\Coh}{{\mathop{\mathsf Coh}}}
\nc{\Ab}{{\mathop{\mathcal{A}\mathit{b}}}}
\nc{\Ccoh}{{\mathop{\mathsf Ccoh}}}
\nc{\Qcoh}{{\mathop{\mathsf Qcoh}}}
\nc{\At}{{\mathop{\mathsf{At}}\nolimits}}
\nc{\tra}{{\mathsf{T}}}
\nc{\fsl}{{\mathfrak{sl}}}
\nc{\fso}{{\mathfrak{so}}}
\nc{\fgl}{{\mathfrak{gl}}}

\nc{\AAV}{{\mathcal{AAV}}}

\nc{\Rep}{{\mathsf{Rep}}}

\nc{\Cubics}{{{\mathcal{S}}_3}}
\nc{\VFT}{{{\mathcal{S}}_{14}}}
\nc{\VFTE}{{{\mathcal{N}}_{\mathrm{reg,sm}}}}
\nc{\MX}{{\CM_X}}
\nc{\MY}{{\CM_Y}}
\nc{\MYE}{{\CM_{Y,\CE}}}
\nc{\Yd}{{Y_d}}
\nc{\Yfive}{{Y_5}}
\nc{\Xg}{{X_{2g-2}}}
\nc{\Xtt}{{X_{22}}}
\nc{\Xst}{{X_{16}}}
\nc{\Xtw}{{X_{12}}}
\nc{\Xe}{{X_{8}}}
\nc{\Xf}{{X_{4}}}

\nc{\git}{{/\!\!/\!{}_\chi}}

\nc{\HOH}{{\mathsf H\mathsf H}}
\nc{\HHE}{{\mathsf H\mathsf E}}

\theoremstyle{plain}

\newtheorem{theorem}{Theorem}[section]
\newtheorem{conjecture}[theorem]{Conjecture}
\newtheorem{lemma}[theorem]{Lemma}
\newtheorem{proposition}[theorem]{Proposition}
\newtheorem{corollary}[theorem]{Corollary}

\theoremstyle{definition}

\newtheorem{definition}[theorem]{Definition}

\theoremstyle{remark}

\newtheorem{remark}[theorem]{Remark}

\title{Hochschild homology and semiorthogonal decompositions}
\author{Alexander Kuznetsov}
\address{\sloppy
\parbox{0.9\textwidth}{
Algebra Section, Steklov Mathematical Institute,
8 Gubkin str., Moscow 119991 Russia
\hfill\\[5pt]
The Poncelet Laboratory, Independent University of Moscow
\hfill
}\bigskip}
\email{akuznet@mi.ras.ru}
\date{}
\thanks{I was partially supported by
RFFI grants 08-01-00297, 07-01-00051, and 07-01-92211,
INTAS 05-1000008-8118,
and the Russian Science Support Foundation.}

\begin{document}

\begin{abstract}
We investigate Hochschild cohomology and homology of admissible subcategories of 
derived categories of coherent sheaves on smooth projective varieties. We show that
the Hochschild cohomology of an admissible subcategory is isomorphic to the derived
endomorphisms of the kernel giving the corresponding projection functor, and 
the Hochschild homology is isomorphic to derived morphisms from this kernel to
its convolution with the kernel of the Serre functor. We investigate some basic 
properties of Hochschild homology and cohomology of admissible subcategories.
In particular, we check that the Hochschild homology is additive with respect
to semiorthogonal decompositions and construct some long exact sequences 
relating the Hochschild cohomology of a category and its semiorthogonal components.
We also compute Hochschild homology and cohomology of some interesting admissible
subcategories, in particular of the nontrivial components of derived categories
of some Fano threefolds and of the nontrivial components of the derived categories
of conic bundles.
\end{abstract}

\maketitle

\section{Introduction}

Cyclic homology and cohomology of schemes was defined by Loday~\cite{L},
Weibel~\cite{W} and Swan~\cite{S}. In the case of a smooth projective variety
they coincide with the Hochschild homology and cohomology and enjoy many pleasant
properties which were investigated by Markarian in~\cite{Ma,Ma2}. By definition
$$
\HOH^\bullet(X) = \Hom^\bullet_{X\times X}(\Delta_*\CO_X,\Delta_*\CO_X),
\qquad
\HOH_\bullet(X) = \HH^\bullet(X\times X,\Delta_*\CO_X\otimes\Delta_*\CO_X),
$$
where the tensor product is considered in the derived sense,
and $\HH^\bullet$ stands for the hypercohomology. Also it is very easy to show that
$\HOH^\bullet(X) = \Hom^\bullet_{X\times X}(\Delta_*\CO_X,\Delta_*\omega_X[\dim X])$.

Thus defined, Hochschild homology and cohomology provide an important
connection between algebra and geometry. On one hand, they generalize
the notion of Hochschild (co)homology of algebras. On the other hand,
as it was shown in \emph{loc.\ cit.} $\HOH_\bullet(X)$ is isomorphic
to the Hodge cohomology of $X$ (with a shifted grading) and $\HOH^\bullet(X)$
is isomorphic to the cohomology of polyvector fields on $X$.
At the same time it is well known that any equivalence $\D^b(X) \cong \D^b(Y)$
induces isomorphisms $\HOH^\bullet(X) \cong \HOH^\bullet(Y)$ and
$\HOH_\bullet(X) \cong \HOH_\bullet(Y)$ (see~\cite{O-HH}).
Thus the Hochschilid cohomology and homology are invariants
of the derived category of coherent sheaves (the natural categorical interpretation
of the Hochschild cohomology is just (a derived version of) the space
of endomorphism of the identity functor of $\D^b(X)$, while for the Hochschild homology
it is the space of self-$\Tor$'s of the identity functor of $\D^b(X)$, or alternatively,
the space of maps from the identity functor of $\D^b(X)$ to its Serre functor).

On the other hand, the Hochschild homology and cohomology is defined for any
differential graded category (see e.g.~\cite{Ke-HH,Ke-HC,Ke}) and hence for any triangulated category
which admits a DG-enhancement. It is also well known that these two definitions agree
(for homology it was shown by Keller in~\cite{Ke-HHS} and for the cohomology it follows
easily from the results of To\"en~\cite{T}).

The goal of the present paper is an investigation of Hochschild homology and cohomology
of a certain class of triangulated categories which is very important for the algebraic
geometry --- admissible subcategories of derived categories of coherent sheaves.
These categories arise as components of semiorthogonal decompositions.



The notions of an admissible subcategory and of a semiorthogonal decomposition
were introduced by Bondal and Kapranov in~\cite{BK}. Roughly speaking, a pair of
strictly full triangulated subcategories $\CA$, $\CB$ of a triangulated
category $\CT$ gives a semiorthogonal decomposition if $\Hom(\CB,\CA) = 0$
and each object $T$ of $\CT$ can be included in a distinguished triangle
of the form $B \to T \to A$ with $A \in \CA$, $B \in \CB$. We write
$\CT = \langle\CA,\CB\rangle$ to indicate such a semiorthogonal decomposition.
In this case $\CA$ is left admissible and $\CB$ is right admissible.
There is a similar notion of a semiorthogonal decomposition with more
than two components, $\CT = \langle \CA_1,\dots,\CA_n\rangle$,
see Section~\ref{prelim} for details.

To define the Hochschild homology and cohomology of an admissible subcategory
$\CA \subset \D^b(X)$ we use a natural DG-enhancement of these categories.
We choose a generator $\CE$ for $\D^b(X)$ and take its component $\CE_\CA$ in $\CA$.
Then it is easy to see that $\CE_\CA$ is a generator of $\CA$ and $\CA$ is equivalent
to the category of perfect complexes over the DG-algebra $C^\bullet = \RHom^\bullet(\CE_\CA,\CE_\CA)$.
Thus, the Hochschild homology and cohomology of $\CA$ are defined as those of $C^\bullet$.

Our first result is a DG-algebra-free interpretation of Hochschild homology and cohomology
in the spirit of the Swan--Weibel definition for schemes. For this we use the results
of~\cite{K-FBC}, where it is proved that the projection functor $\D^b(X) \to \CA$
is representable by a kernel $P \in \D^b(X\times X)$. We show that
$$
\HOH^\bullet(\CA) \cong \Hom^\bullet_{X\times X}(P,P),
\qquad
\HOH_\bullet(\CA) \cong \HH^\bullet(X\times X,P\otimes P^\tra) \cong \Hom^\bullet(P,P\otimes p_2^*\omega_X[\dim X]),
$$
where $P^\tra$ is the transposed projection kernel (the pullback of $P$ under the transposition
of factors map $X\times X \to X\times X$).

Being motivated by these formulas we define {\em the generalized Hochschild cohomology of $X$
with support in $T \in \D^b(X\times X)$ and coefficients in $E \in \D^b(X\times X)$}\/ as
$$
\HOH^\bullet_T(X,E) = \Hom^\bullet(E,E\circ T),
$$
where $E\circ T$ is the convolution of $E$ and $T$ considered as kernels.
In the particular case of $T = \Delta_*\CO_X$, $E = P$ they give
the Hochschild cohomology of $\CA$, and in the case $T = \Delta_*\omega_X[\dim X]$, $E = P$
they give the Hochschild homology of $\CA$.
We investigate some general properties of generalized Hochschild cohomology,
especially their functoriality and behavior with respect to changing the coefficients
and supports. In particular, we show that any kernel functor $\Phi_K:\D^b(X) \to \D^b(Y)$ induces
a map on generalized Hochschild cohomology if the kernel $K$
satisfies some compatibility conditions.

Further, we go back to the Hochshild homology and cohomology of admissible subcategories.
As it follows easily from the definition any equivalence $\CA \cong \CB$ of admissible
subcategories $\CA \subset \D^b(X)$, $\CB \subset \D^b(Y)$ should give an isomorphism
of their Hochschild (co)homology (certainly, for this the equivalence should be compatible
with DG-enhancements). However, it is not so easy to construct an explicit isomorphism.
We use the formalism and functoriality of generalized Hochschild homology to give explicit
isomorphisms.

We also address a question of the relation between the Hochschild (co)homology of a scheme $X$
and those of semiorthogonal components of $\D^b(X)$. In case of homology, using the DG-approach
it is easy to argue that whenever $\D^b(X) = \langle \CA_1, \dots, \CA_n \rangle$, there is an isomorphism
$$
\HOH_\bullet(X) \cong \bigoplus_{i=1}^n \HOH_\bullet(\CA_i).
$$
We show that in this case we have a direct sum decomposition in the most rigid sense ---
for each admissible subcategory $\CA \subset \D^b(X)$ its Hochschild homology can be
identified with a canonical vector subspace $\HOH_\bullet(\CA) \subset \HOH_\bullet(X)$ and
for a semiorthogonal decomposition the Hochschild homology of $X$ decomposes
into the direct sum of the corresponding subspaces.

The situation with the Hochschild cohomology is much more complicated.
In this case we don't have any additivity. Instead, we argue that
if $\CA = \langle \CA_1,\CA_2 \rangle$ is a 2-term semiorthogonal decomposition
then there is a long exact sequence
$$
\dots \to \HOH^i(\CA) \to \HOH^i(\CA_1) \oplus \HOH^i(\CA_2) \to \Ext^i(\phi,\phi) \to \HOH^{i+1}(\CA) \to \dots,
$$
where $\phi:\CA_1 \to \CA_2$ is the gluing functor.

Finally, we compute the Hochschild homology and cohomology of some interesting
admissible subcategories, such as
$\langle \CO_X \rangle^\perp \subset \D^b(X)$, if $\CO_X$ is an exceptional bundle (e.g.\ if $X$ is a Fano variety),
$\langle E,\CO_X \rangle^\perp \subset \D^b(X)$, if $(E,\CO_X)$ is an exceptional pair,
and $(f^*\D^b(Y))^\perp \subset \D^b(X)$, if $f:X \to Y$ is a conic bundle.

We conclude the paper with a short discussion of a Nonvanishing Conjecture
for Hochschild homology.

\medskip

The paper is organized as follows. In Section~\ref{prelim} we introduce
the technical notions used in the paper, such as Hochschild homology and cohomology
of algebraic varieties, semiorthogonal decompositions and mutations.
In Section~\ref{kernels} we remind some standard results on the calculus of kernels,
sketch the results of~\cite{K-FBC} on representability of the projection functors
and investigate the relations between the projection kernels for a semiorthogonal
decomposition, which are essential for the further treatment.
In Section~\ref{sec-dg} we construct a DG-enhancement for any
admissible subcategory $\CA \subset \D^b(X)$, define the Hochschild
homology and cohomology of $\CA$, and compute it in terms of the kernel
of the corresponding projection functor.
In Section~\ref{ghoh} we define the generalized Hochschild cohomology and investigate
its properties with respect to change of supports and coefficients.
In Section~\ref{ghohfun} we investigate the action of kernel functors
on generalized Hochschild cohomology.
In Section~\ref{hohadm} we construct explicit isomorphisms between Hochschild homology
and cohomology of equivalent admissible subcategories and establish the additivity
of Hochschild homology and long exact sequences for the Hochschild cohomology
of an admissible category and of its semiorthogonal components.
In Section~\ref{s-comp} we compute the Hochschild homology
and cohomology of some interesting admissible subcategories.
Finally, in Section~\ref{sec-nvc} we discuss the nonvanishing conjecture
for Hochschild homology.


{\bf Acknowledgements:}
I thank  A.Bondal, D.Kaledin, and D.Orlov for many useful discussions.
I am very grateful to N.Markarian for careful reading of the draft version
of this paper and helpful comments.

\section{Preliminaries}\label{prelim}

\subsection{Notations}

The base field $\kk$ is assumed to be of zero characteristic.

Given an algebraic variety $X$ we denote by $\D^b(X)$ the bounded derived category of
coherent sheaves on $X$. Given a morphism $f:X \to Y$ we denote by $f_*$ and $f^*$
the {\em total}\/ derived pushforward and
the {\em total}\/ derived pullback functors.
The twisted pullback functor is denoted by $f^!$
(it is right adjoint to $f_*$ if $f$ is proper).
Similarly, $\otimes$ stands for the derived tensor product,
and $\RHom$, $\RCHom$ stand for the global and local $\RHom$ functors.
For an object $F \in \D^b(X)$ we put $F^\vee := \RCHom(F,\CO_X)$.
We denote by $\omega_X$ the canonical line bundle of $X$.

For a functor $\Phi$ between triangulated categories we denote by
$\Phi^*$ and $\Phi^!$ the left and the right adjoint functors to $\Phi$
(when they exist).

Given an algebraic variety $X$ we denote by $\Delta:X \to X\times X$ the diagonal
embedding. We also denote
$$
S_X = \Delta_*\omega_X[\dim X], \qquad
S^{-1}_X = \Delta_*\omega^{-1}_X[- \dim X] \qquad \in \D^b(X\times X)
$$
For a subcategory $\CA$ of a triangulated category $\CT$ we denote by $\CA^\perp$
and ${}^\perp\CA$ the right and the left orthogonal to $\CA$ defined by
$$
\CA^\perp = \{ T \in \CT \ |\ \forall A \in \CA\ \Hom(A[i],T) = 0 \},\qquad
{}^\perp\CA = \{ T \in \CT \ |\ \forall A \in \CA\ \Hom(T,A[i]) = 0 \}.
$$
These are triangulated subcategories of $\CT$ closed under taking direct summands.


\subsection{Hochschild cohomology of algebraic varieties}

Let $X$ be a smooth projective variety.
The Hochschild cohomology and homology of $X$ are defined by
\begin{equation}\label{hohx}
\HOH^\bullet(X) = \Hom^\bullet_{X\times X}(\Delta_*\CO_X,\Delta_*\CO_X),
\qquad
\HOH_\bullet(X) = \HH^\bullet(X\times X,\Delta_*\CO_X\otimes\Delta_*\CO_X)
\end{equation}
(see~\cite{S,W,L}).
For our purposes another definition of Hochschild homology is more convenient.

\begin{lemma}\label{hohserre}
There is a canonical isomorphism
$\HH^\bullet(X\times X,\Delta_*\CO_X\otimes\Delta_*\CO_X) \cong
\Hom^\bullet_{X\times X}(\Delta_*\CO_X,S_X)$.
\end{lemma}
\begin{proof}
We have
$$
\HH^\bullet(X\times X,\Delta_*\CO_X\otimes\Delta_*\CO_X) \cong
\Hom^\bullet_{X\times X}(\CO_{X\times X},\Delta_*\CO_X\otimes\Delta_*\CO_X) \cong
\Hom^\bullet_{X\times X}((\Delta_*\CO_X)^\vee,\Delta_*\CO_X).
$$
On the other hand, by Grothedieck duality we have
$$
(\Delta_*\CO_X)^\vee =
\RCHom(\Delta_*\CO_X,\CO_{X\times X}) \cong
\Delta_*\RCHom(\CO_X,\Delta^!(\CO_{X\times X})) \cong
\Delta_*\omega_X^{-1}[-\dim X],
$$
so we conclude that
$$
\HOH_\bullet(X) \cong
\Hom^\bullet_{X\times X}(\Delta_*(\omega_X^{-1}[-\dim X]),\Delta_*\CO_X) \cong
\Hom^\bullet_{X\times X}(\Delta_*\CO_X,S_X),
$$
the last isomorphism is obtained by tensoring with $\omega_X[\dim X]$.
\end{proof}

So, we will use this as a definition of Hochschild homology
\begin{equation}\label{hohx1}
\HOH_\bullet(X) = \Hom^\bullet_{X\times X}(\Delta_*\CO_X,S_X).
\end{equation}

\begin{remark}
In addition to the canonical isomorphism of Lemma~\ref{hohserre}
there is also a canonical duality between the spaces
$H^\bullet(X\times X,\Delta_*\CO_X\otimes\Delta_*\CO_X)$ and
$\Hom^\bullet_{X\times X}(\Delta_*\CO_X,S_X)$. So, usually the space
$\Hom^\bullet_{X\times X}(\Delta_*\CO_X,S_X)$ is considered as the dual
to the Hochschild homology of $X$. However, we prefer to use the identification
of Lemma~\ref{hohserre} and consider it as the Hochschild homology of $X$ itself.
\end{remark}

\subsection{Semiorthogonal decompositions and projection functors}

\begin{definition}[\cite{BK,BO1,BO2}]
A {\sf semiorthogonal decomposition}\/ of a triangulated category $\CT$ is a sequence
of full triangulated subcategories
$\CA_1,\dots,\CA_m$ in $\CT$ such that $\CA_i \subset \CA_j^\perp$ for $i < j$
and for every object $T \in \CT$ there exists a chain of morphisms
$0 = T_m \to T_{m-1} \to \dots \to T_1 \to T_0 = T$ such that
the cone of the morphism $T_k \to T_{k-1}$ is contained in $\CA_k$
for each $k=1,2,\dots,m$. In other words, there exists a diagram
\begin{equation}\label{tower}
\vcenter{
\xymatrix@C-7pt{
0 \ar@{=}[r] & T_m \ar[rr]&& T_{m-1} \ar[dl] \ar[rr]&& \quad\dots\quad \ar[rr]&& T_2 \ar[rr]&& T_1 \ar[dl] \ar[rr]&& T_0 \ar[dl] \ar@{=}[r] & T \\
&& A_m \ar@{..>}[ul] &&& \dots &&& A_2 \ar@{..>}[ul]&& A_1 \ar@{..>}[ul]&&
}}
\end{equation}
where all triangles are distinguished (dashed arrows have degree $1$) and $A_k \in \CA_k$.
\end{definition}

Thus, every object $T\in\CT$ admits a decreasing ``filtration''
with factors in $\CA_1$, \dots, $\CA_m$ respectively.
The following properties of semiorthogonal decompositions are well-known.

\begin{lemma}\label{ff}
If $\CT = \lan \CA_1, \dots, \CA_m \ran$ is a semiorthogonal decomposition and $T \in \CT$
then the diagram~\eqref{tower} for $T$ is unique and functorial\/ {\rm(}for any morphism $T \to T'$
there exists a unique collection of morphisms $T_i \to T'_i$, $A_i \to A'_i$ combining into
a morphism of diagram~\eqref{tower} for $T$ into diagram~\eqref{tower} for $T'${\rm)}.
\end{lemma}
\begin{proof}
Note that $T_1 \in \lan \CA_2, \dots, \CA_m \ran$ by~\eqref{tower}.
It follows from the semiorthogonality property that $\Hom(T_1,A'_1[k]) = 0$ for all $k \in \ZZ$.
Therefore any map $T_0 = T \to T' = T'_0$ extends in a unique way to a map of the triangle
$T_1 \to T_0 \to A_1$ into the triangle $T'_1 \to T'_0 \to A'_1$. In particular, we obtain
a map $T_1 \to T'_1$ and proceed by induction.
\end{proof}

We denote by $\alpha_k:\CT \to \CT$ the functor $T \mapsto A_k$ and by $\tau_i:\CT \to \CT$
the functor $T \mapsto T_i$. We call $\alpha_k$ the {\sf $k$-th projection functor}\/
and $\tau_k$ the {\sf $k$-th truncation functor}\/ of the semiorthogonal decomposition.

\begin{definition}[\cite{BK,B}]
A full triangulated subcategory $\CA$ of a triangulated category $\CT$ is called
{\sf right admissible}\/ if for the inclusion functor $i:\CA \to \CT$ there is
a right adjoint $i^!:\CT \to \CA$, and
{\sf left admissible}\/ if there is a left adjoint $i^*:\CT \to \CA$.
Subcategory $\CA$ is called {\sf admissible}\/ if it is both right and left admissible.
\end{definition}

\begin{lemma}[\cite{B}]\label{sod_adm}
If $\CT = \lan\CA,\CB\ran$ is a semiorthogonal decomposition then
$\CA$ is left admissible and $\CB$ is right admissible.
Conversely, if $\CA \subset \CT$ is left admissible then
$\CT = \lan \CA,{}^\perp\CA \ran$ is a semiorthogonal decomposition, and
if $\CB \subset \CT$ is right admissible then
$\CT = \lan \CB^\perp,\CB \ran$ is a semiorthogonal decomposition.
\end{lemma}

\begin{lemma}[\cite{BK}]\label{sod_adm1}
If $\D^b(X) = \lan\CA_1,\dots,\CA_m\ran$ is a semiorthogonal decomposition and $X$
is smooth and projective then all subcategories $\CA_i \subset \D^b(X)$ are admissible.
\end{lemma}

If $\CA \subset \CT$ is an admissible subcategory, the projection functor onto $\CA$
with respect to the semiorthogonal decomposition $\CT = \langle \CA^\perp,\CA \rangle$
(resp.\ $\CT = \lan \CA,{}^\perp\CA \ran$) is called {\sf the right}\/ (resp.\ {\sf the left})
{\sf projection functor onto~$\CA$}.

Let $X$ be an algebraic variety over a smooth base scheme $S$, $f:X \to S$ being the structure morphism.
A triangulated subcategory $\CA \subset \D^b(X)$ is called {\sf $S$-linear}\/ if for any objects $F \in \CA$, $G \in \D^b(S)$
we have $F \otimes f^*G \in \CA$. A semiorthogonal decomposition $\D^b(X) = \langle \CA_1,\dots,\CA_m \rangle$
is {\sf $S$-linear}\/ if all its components $\CA_i$ are $S$-linear.

\begin{lemma}[\cite{K-HS}]
If $\CA$ is a left {\rm(}resp. right{\rm)} admissible $S$-linear subcategory in $\D^b(X)$ then
the subcategory ${}^\perp\CA$ {\rm(}resp. $\CA^\perp${\rm)} is also $S$-linear.
\end{lemma}

\subsection{Mutations}\label{ss-mut}

Given a semiorthogonal decomposition of the derived category of a smooth projective variety
one can produce many other decompositions. Actually, there is an action of the braid group
on $n$ strands on the set of all $n$-term semiorthogonal decompositions. The action is given
by mutations. Here we remind the necessary constructions.

Let $\CT = \langle \CA_1,\dots,\CA_m \rangle$ be a semiorthogonal decomposition with all $\CA_i$ being admissible
(the last condition holds automatically if $\CT = \D^b(X)$ for $X$ being smooth and projective).
Define
$$
\begin{array}{l}
\RR_i(\CA_\bullet)_j =
\begin{cases}
\CA_j, & \text{if $j \ne i-1,i$},\\
\CA_i, & \text{if $j = i-1$},\\
{}^\perp\langle \CA_1,\dots,\CA_{i-2},\CA_i \rangle \cap \langle \CA_{i+1},\dots,\CA_m \rangle^\perp, & \text{if $j = i$},
\end{cases}\\
\LL_i(\CA_\bullet)_j =
\begin{cases}
\CA_j, & \text{if $j \ne i,i+1$},\\
{}^\perp\langle \CA_1,\dots,\CA_{i-1} \rangle \cap \langle \CA_i,\CA_{i+2},\dots,\CA_m \rangle^\perp, & \text{if $j = i$},\\
\CA_i, & \text{if $j = i+1$}.
\end{cases}\
\end{array}
$$

\begin{theorem}[\cite{BK}]\label{muts}
If $\CT = \langle \CA_1,\dots,\CA_m \rangle$ is a semiorthogonal decomposition with all $\CA_i$ being admissible
then for each $i \ge 2$ there is a semiorthogonal decomposition
$\CT = \langle \RR_i(\CA_\bullet)_1,\dots,\RR_i(\CA_\bullet)_m \rangle$
and for each $i \le m-1$ there is a semiorthogonal decomposition
$\CT = \langle \LL_i(\CA_\bullet)_1,\dots,\LL_i(\CA_\bullet)_m \rangle$.
Moreover, the braid group relations are satisfied: $\RR_i\RR_{i+1}\RR_i = \RR_{i+1}\RR_i\RR_{i+1}$ and
$\LL_i\LL_{i-1}\LL_i = \LL_{i-1}\LL_i\LL_{i-1}$ for each $2 \le i \le m-1$.
\end{theorem}

The operations $\RR_i$ and $\LL_i$ are known as the {\sf right}\/ and {\sf left mutations}\/ respectively.

Actually, the components of the decomposition obtained by any mutation don't change.
\begin{theorem}[\cite{BK}]\label{mut-fun}
Let\/ $\BR_i$ be the right projection functor onto the subcategory ${}^\perp\CA_i$
and\/ $\BL_i$ be the left projection functor onto $\CA_i^\perp$.
Then $\BR_i$ and $\BL_i$ induce equivalences
$$
\BR_i:\CA_{i-1} \to \RR_i(\CA_\bullet)_i,
\qquad
\BL_i:\CA_{i+1} \to \LL_i(\CA_\bullet)_i.
$$
Moreover, for any object $T \in \CA_{i-1}$ there is a distinguished triangle
$\BR_i(T) \to T \to \alpha\alpha^*(T)$,
and for any $T \in \CA_{i+1}$ there is a distinguished triangle
$\alpha\alpha^!(T) \to T \to \BL_i(T)$,
where $\alpha:\CA_i \to \CT$ is the embedding functor.
\end{theorem}

Consider the following sequence of mutations
$$
\DD = \prod_{k=2}^{m}(\RR_m\dots\RR_{k+1}\RR_k) =
(\RR_m\dots\RR_2)(\RR_m\dots\RR_3)(\RR_m\RR_{m-1})\RR_m.
$$
It is easy to see that it takes a semiorthogonal decomposition $\CT = \langle \CA_1,\dots,\CA_m \rangle$
to the decomposition $\CT = \langle \CB_m,\dots,\CB_1 \rangle$, where
\begin{equation}\label{bi}
\CB_i := {}^\perp \langle \CA_1,\dots,\CA_{i-1},\CA_{i+1},\dots,\CA_m \rangle,
\end{equation}
and it follows from Theorem~\ref{mut-fun}
that the functor $\BD_i = \BR_m\BR_{m-1}\dots\BR_{i+1}$ gives an equivalence $\CA_i \to \CB_i$.

\begin{lemma}\label{bdi}
For any $T \in \CA_i$ we have $\Hom^\bullet(\BD_i(T),T) \cong \Hom^\bullet(T,T)$.
\end{lemma}
\begin{proof}
By Theorem~\ref{mut-fun} we have a distinguished triangle
$$
\BD_i(T) \to T \to T',
$$
where $T' \in \langle \CA_{i+1},\dots,\CA_m \rangle$. We have $\Hom^\bullet(T',T) = 0$ by semiorthogonality,
hence we have an isomorphism $\Hom^\bullet(\BD_i(T),T) \cong \Hom^\bullet(T,T)$.
\end{proof}

The semiorthogonal decomposition $\CT = \langle \CB_m,\dots,\CB_1 \rangle$ is known as
{\sf the dual} for $\CT = \langle \CA_1,\dots,\CA_m \rangle$.

\begin{lemma}[\cite{BK}]\label{dual-mut}
We have $\DD\circ\LL_i = \LL_{m+1-i}\circ\DD$, $\DD\circ\RR_i = \RR_{m+1-i}\circ\DD$.
\end{lemma}

It follows in particular that $\DD^2$ is the central element
in the braid group. Actually, it is well known that $\DD^2$ acts as the inverse of the Serre functor.

\section{The calculus of kernels}\label{kernels}

\subsection{Kernel functors}

Most part of the material of this section is standard.
Some of the statements we remind for the convenience of the reader,
some we make a little bit more precise than one can find in the literature.

In this section whenever we consider an algebraic variety $X$
over a base scheme $S$ we always assume that both $X$ and $S$ are smooth,
and the structure map $X \to S$ is proper and flat (although the morphism $X \to S$
doesn't need to be smooth).

For any collection of algebraic varieties $X_1$, \dots, $X_n$ over the same base scheme $S$
and any subset $I = \{i_1,\dots,i_m\} \subset \{1,\dots,n\}$
we denote by $p_I$ the projection $X_1\times_S \dots \times_S  X_n \to X_{i_1}\times_S \dots\times_S  X_{i_m}$.
Similarly, for any pair of indexes $i < j$ we denote by $\sigma_{ij}$ the transposition of the $i$-th and $j$-th factors
$X_1\times_S \dots\times_S  X_i\times_S \dots\times_S  X_j\times_S \dots \times_S  X_n \to
X_1\times_S \dots\times_S  X_j\times_S \dots\times_S  X_i\times_S \dots \times_S  X_n$.

For each object $K \in D^b(X_1\times_S X_2)$ we denote by $\Phi_K$ the functor
defined by
$$
\Phi_K(F) = p_{1*}(K \otimes p_2^*(F)),
$$
which is called \emph{the kernel functor} with kernel~$K$.
Note that since $X_2$ is smooth every object $F \in \D^b(X_2)$ is a perfect complex,
so $K\otimes p_2^*(F)$ has bounded coherent cohomology, and since the map $p_1$ is proper
we have $\Phi_K(F) \in \D^b(X_1)$. So, $\Phi_K$ is a functor $\D^b(X_2) \to \D^b(X_1)$.

Note that $\Delta_*\CO_X$ is the kernel of the identity functor of $\D^b(X)$,
and that
$$
S_{X/S} := \Delta_*\omega_{X/S}[\dim X/S]
$$
is the kernel giving the Serre functor of $X$ over $S$ (see~\cite{BK}).

We always consider the kernel functors acting from right to left,
i.e.\ from the derived category of the second factor to that of the first.
If we need a functor in the opposite direction, we use explicitly
the transposition morphism. We denote the transposed kernel by
$$
K^\tra = \sigma_{12}^*K \in \D^b(X_2\times_S X_1).
$$
Thus we have
\begin{equation}\label{tker}
p_{2*}(K \otimes p_1^*(F)) \cong \Phi_{K^\tra}(F)
\end{equation}

Given two kernels $K \in D^b(X_2\times_S X_3)$ and $L \in D^b(X_1\times_S X_2)$
we denote by $L\circ K \in D^b(X_1\times_S X_3)$ their \emph{convolution}.
It is defined by
$$
L\circ K := p_{13*}(p_{12}^*L \otimes p_{23}^*K),
$$
so that $\Phi_{L\circ K} \cong \Phi_L \circ \Phi_K$.

The convolution is associative --- there exists a canonical functorial (in each argument) isomorphism
$$
a_{M,L,K}:M\circ(L\circ K) \to (M\circ L) \circ K,
$$
such that the pentahedron equality holds,
$$
(a_{N,M,L}\circ \id_K)a_{N,M\circ L,K}(\id_N\circ a_{M,L,K}) = a_{N,M,L\circ K}a_{N\circ M,L,K}.
$$


In some cases the convolution is easy to compute.
We will need the following results

\begin{lemma}\label{conv-diag}
For any $E_1 \in \D^b(X_1)$, $E_2 \in \D^b(X_2)$, $K \in \D^b(X_1\times_S X_2)$ we have
$$
(\Delta_{X_1*}E_1)\circ K \cong p_1^*E_1\otimes K,
\qquad
K\circ (\Delta_{X_2*}E_2) \cong K\otimes p_2^*E_2.
$$
\end{lemma}
\begin{proof}
Consider the product $X_1\times_S X_1\times_S X_2$. We have
\begin{multline*}
p_{13*}(p_{12}^*(\Delta_{X_1*}E_1) \otimes p_{23}^*K) \cong
p_{13*}(\tilde\Delta_*p_1^*E_1 \otimes p_{23}^*K) \cong
p_{13*}\tilde\Delta_*(p_1^*E_1 \otimes \tilde\Delta^*p_{23}^*K) \cong
p_1^*E_1 \otimes \tilde\Delta^*p_{23}^*K,
\end{multline*}
where $\tilde\Delta:X_1\times_S X_2 \to X_1\times_S X_1\times_S X_2$ is the product of the diagonal and the identity
(the first isomorphism is the base change, the second is the projection formula, and the third uses equalities
$p_{13}\circ\tilde\Delta = p_{23}\circ\tilde\Delta = \id$). The second claim is proved analogously.
\end{proof}

\begin{lemma}\label{conv-decomp}
For any $E_1 \in \D^b(X_1)$, $E_2 \in \D^b(X_2)$, $K \in \D^b(X_2\times_S X_3)$ we have
$$
(p_1^*E_1\otimes p_2^*E_2) \circ K \cong p_1^*E_1\otimes p_2^*\Phi_{K^\tra}(E_2).
$$
\end{lemma}
\begin{proof}
Indeed,
\begin{multline*}
p_{13*}(p_{12}^*(p_1^*E_1\otimes p_2^*E_2) \otimes p_{23}^*K) \cong
p_{13*}(p_{13}^*(p_1^*E_1)\otimes p_{23}^*(p_1^*E_2 \otimes K)) \cong \\
p_1^*E_1\otimes p_{13*}p_{23}^*(p_1^*E_2 \otimes K) \cong
p_1^*E_1\otimes p_2^*p_{2*}(p_1^*E_2 \otimes K) \cong
p_1^*E_1\otimes p_2^*\Phi_{K^\tra}(E_2)
\end{multline*}
(the first isomorphism is evident, the second is the projection formula, \
the third is the base change, and the last is~\eqref{tker}).
\end{proof}

It is well known that every kernel functor between smooth projective varieties has
a left and a right adjoint functors, which are kernel functors as well.
Below we will need a more precise statement.

\begin{lemma}\label{adjs}
If $K \in \D^b(X_1\times_S X_2)$ and $\Phi_K:\D^b(X_2) \to \D^b(X_1)$ is the corresponding kernel functor
then the adjoint functors of $\Phi_K$ are kernel functors as well. Explicitly, $\Phi_K^* \cong \Phi_{K^*}$
and $\Phi_K^! \cong \Phi_{K^!}$, where
$$
K^* = \RCHom(K^\tra,p_2^*\omega_{X_1/S}[\dim X_1/S]),\qquad
K^! = \RCHom(K^\tra,p_1^*\omega_{X_2/S}[\dim X_2/S]).
$$
Moreover, the adjunction maps $\id_{X_1} \to \Phi_K\Phi_K^*$ and $\Phi_K\Phi_K^! \to \id_{X_1}$
are induced by some natural maps $\lambda_K:\Delta_*\CO_{X_1} \to K \circ K^*$ and $\rho_K:K\circ K^! \to \Delta_*\CO_{X_1}$ respectively.
\end{lemma}
\begin{proof}
The formulas for the kernels of the adjoint functors are well-known (see e.g.~\cite{BO1}),
so we only have to construct the morphisms $\lambda_K$ and $\rho_K$.
Let $d_i = \dim X_i/S$. Consider the following base change diagram
$$
\xymatrix{
X_1\times_S X_2 \ar[rr]^{\tilde\Delta} \ar[d]_{p_1} && X_1 \times_S X_2 \times_S X_1 \ar[d]^{p_{13}} \\
X_1 \ar[rr]^{\Delta} && X_1\times_S X_1
}
$$
where $\tilde\Delta$ is the product of the diagonal on $X_1$ with the identity on $X_2$.
Note that
\begin{multline*}
\Delta^!(K\circ K^*) =
\Delta^!p_{13*}(p_{12}^*K\otimes p_{23}^*\RCHom(K^\tra,p_2^*\omega_{X_1/S}[d_1])) \cong \\ \cong
p_{1*}\tilde\Delta^!(p_{12}^*K\otimes p_{23}^*\RCHom(K^\tra,p_2^*\omega_{X_1/S}[d_1])) \cong \\ \cong
p_{1*}(\tilde\Delta^*p_{12}^*K\otimes \tilde\Delta^*p_{23}^*\RCHom(K^\tra,p_2^*\omega_{X_1/S}[d_1])\otimes p_1^*\omega_{X_1/S}^{-1}[-d_1]) \cong \\ \cong
p_{1*}(K\otimes \sigma_{12}^*\RCHom(K^\tra,p_2^*\omega_{X_1/S}[d_1])\otimes p_1^*\omega_{X_1/S}^{-1}[-d_1]) \cong \\ \cong
p_{1*}(K\otimes \RCHom(K,p_1^*\omega_{X_1/S}[d_1])\otimes p_1^*\omega_{X_1/S}^{-1}[-d_1]) \cong \\ \cong
p_{1*}\RCHom(K,K\otimes p_1^*\omega_{X_1/S}[d_1]\otimes p_1^*\omega_{X_1/S}^{-1}[-d_1]) \cong
p_{1*}\RCHom(K,K)
\end{multline*}
(the first isomorphism is the definition of the convolution, the second follows
from the faithful base change theorem~\cite{K-HS}, the third is the definition
of the twisted pullback). Further, the canonical map
$p_1^*\CO_{X_1} \cong \CO_{X_1\times_S X_2} \to \RCHom(K,K)$ induces a map
$\CO_{X_1} \to p_{1*}\RCHom(K,K)$, and by adjunction we get a map $\lambda_K:\Delta_*\CO_{X_1} \to K\circ K^*$.
The map $\rho_K$ is constructed analogously.
%
\end{proof}

We will also need a well-known relation of the Serre functor with the left and the right adjoints,
expressed in terms of kernels. By Lemma~\ref{conv-diag} and Lemma~\ref{adjs} we have
\begin{multline*}
S_{X_2}\circ K^* \cong
\RCHom(K^\tra,p_2^*\omega_{X_1/S}[\dim X_1/S])\otimes p_1^*\omega_{X_2/S}[\dim X_2/S] \cong \\ \cong
\RCHom(K^\tra,p_1^*\omega_{X_2/S}[\dim X_2/S])\otimes p_2^*\omega_{X_1/S}[\dim X_1/S] \cong
K^!\circ S_{X_1}.
\end{multline*}
Thus we have

\begin{corollary}\label{sks}
For any kernel $K$ we have an isomorphism
$$
S_{X_2}\circ K^* \cong K^!\circ S_{X_1}
$$
which is functorial in $K$.
\end{corollary}

A straightforward computation shows that we have the following result.

\begin{lemma}\label{adj-diag}
Assume that $K = \Delta_*E \in \D^b(X\times_S X)$ for some $E \in \D^b(X)$.
Then $K^* \cong K^! \cong \Delta_*(E^\vee)$.
\end{lemma}

We also will need the following result.

\begin{lemma}\label{tr-ker}
Let $K \in \D^b(X_1\times_S X_2)$. Then $\Phi_{K^\tra}(F^\vee)^\vee \cong \Phi_K^*(F)$.
\end{lemma}
\begin{proof}
Indeed,
\begin{multline*}
\Phi_{K^\tra}(F^\vee)^\vee =
\RCHom(p_{2*}(K\otimes p_1^*F^\vee),\CO_{X_2}) \cong
p_{2*}\RCHom(K\otimes p_1^*F^\vee,p_2^!\CO_{X_2}) \cong \\
p_{2*}(K^\vee\otimes p_1^*\omega_{X_1/S}[\dim X_1/S]\otimes p_1^*F) \cong
\Phi_{{K^\tra}^\vee\otimes p_1^*\omega_{X_1/S}[\dim X_1/S]}(F) =
\Phi_{K^*}(F)
\end{multline*}
and this is precisely what we need.
\end{proof}

\subsection{Kernels of the projection functors}

We start with a reminder on some results of~\cite{K-FBC}.

Let $X$ be an algebraic variety over a base scheme $S$ and assume that $\D^b(X) = \langle \CA_1,\dots,\CA_m \rangle$
is an $S$-linear semiorthogonal decomposition. Let $\phi:T \to S$ be a base change and denote $X_T = X\times_S T$.
Then under some technical conditions it is proved in~\cite{K-FBC} that there is a semiorthogonal decomposition
$\D^b(X_T) = \langle \CA_{1T},\dots,\CA_{mT} \rangle$ which is compatible with the initial decomposition
with respect to the functors $\phi_*$ and $\phi^*$. The construction of the subcategories $\CA_{iT} \subset \D^b(X_T)$
is rather complicated in general --- first one should consider the induced decomposition
of the category of perfect complexes on $X$, then pull it back to obtain a decomposition
of the category of perfect complexes on $X_T$, then extend it to a decomposition of $\D^-(X_T)$,
and finally obtain a decomposition of $\D^b(X_T)$ by taking the intersection.
However, in the case when both $X$ and $T$ are smooth, one can write down
a closed formula for $\CA_{iT}$
$$
\CA_{iT} = \langle\langle \phi^*A\otimes p^*F \rangle\rangle_{A \in \CA_i,\ F\in\D^b(T)},
$$
where $\langle\langle - \rangle\rangle$ denotes the extension of a category by adding
all homotopy colimits which are bounded and coherent (note that this includes adding all direct summands).
%

Below we will need a special case of this result. Assume that $X$ is smooth. Consider the square
$$
\xymatrix{
X\times_S X \ar[r] \ar[d] & X \ar[d] \\ X \ar[r] & S
}
$$
as a base change diagram in two ways --- taking either the vertical or the horizontal arrows
to be the base change. This gives two semiorthogonal decompositions for $\D^b(X\times_S X)$
$$
\D^b(X\times_S X) = \langle \CA_{1X},\dots,\CA_{mX} \rangle = \langle {}_X\CA_1,\dots,{}_X\CA_m \rangle,
$$
where
$$
\CA_{iX} = \langle\langle p_1^*A_i \otimes p_2^*F \rangle\rangle_{A_i\in\CA_i,\ F \in \D^b(X)},
\qquad
{}_X\CA_i = \langle\langle p_1^*F \otimes p_2^*A_i \rangle\rangle_{A_i\in\CA_i,\ F \in \D^b(X)}.
$$
%
These semiorthogonal decompositions allow to prove that the projection
and the truncation functors of a semiorthogonal decomposition can be represented by kernels.

\begin{theorem}[\cite{K-FBC}]\label{sod-kernels}
Let $X$ be a quasiprojective variety over $S$ and $\D^b(X) = \lan \CA_1,\dots,\CA_m \ran$ an $S$-linear semiorthogonal decomposition.
Let $\alpha_i:\D^b(X) \to \D^b(X)$ and $\tau_i:\D^b(X) \to \D^b(X)$ be the projection and the truncation functors
of the semiorthogonal decomposition.
Then for every $i \in \{1,\dots,m\}$ there are objects $P_i,D_i \in \D^b(X\times_S X)$, unique up to an isomorphism,
such that $\alpha_i \cong \Phi_{P_i}$ and $\tau_i \cong \Phi_{D_i}$.
Moreover, the triangles of functors $\tau_i \to \tau_{i-1} \to \alpha_i$
are induced by the distinguished triangles $D_i \to D_{i-1} \to P_i$ in $\D^b(X\times_S X)$.
\end{theorem}

We give a sketch of proof for completeness.

\begin{proof}
Consider the induced semiorthogonal decomposition
$\D^b(X\times_S X) = \langle \CA_{1X},\dots,\CA_{mX} \rangle$.
Let
$$
0 = D_m \to D_{m-1} \to \dots \to D_1 \to D_0 = \Delta_*\CO_X
$$
be the corresponding filtration of $\Delta_*\CO_X$. Then the kernels $D_i$ represent
the truncation functors $\tau_i$ and the kernels $P_i = \Cone(D_i \to D_{i-1}) \in \CA_{iX}$
represent the projection functors. The uniqueness of kernels follows from the uniqueness
of semiorthogonal components of an object with respect to a semiorthogonal decomposition.
\end{proof}

The following results are essential for the rest of the paper.


\begin{proposition}\label{ki}
Let $\D^b(X) = \lan \CA_1,\dots,\CA_m\ran$ be an $S$-linear semiorthogonal decomposition and
let $P_i$ be the kernels of the corresponding projection functors onto $\CA_i$.
Then
$$
P_i \in \CA_i\boxtimes \CB_i^\vee := \CA_{iX} \cap {}_X\CB_i^\vee \subset \D^b(X\times_S X),
$$
where $\CB_i$ is the component of the dual semiorthogonal decomposition defined by~\eqref{bi}.
\end{proposition}
\begin{proof}
According to the proof of Theorem~\ref{sod-kernels} we have $P_i \in \CA_{iX}$,
so it suffices to check that $P_i \in {}_X\CB_i^\vee$. For this we note that
$\Phi_{P_i}(\CA_j) = \alpha_i(\CA_j) = 0$ for $i \ne j$.
In the other words, we have $p_{1*}(P_i\otimes p_2^*(\CA_j)) = 0$.
It follows that for any $F \in \D^b(X)$, $A \in \CA_j$ we have
$$
0 =
\Hom(F,p_{1*}(P_i\otimes p_2^*A)) =
\Hom(p_1^*F,P_i\otimes p_2^*A) =
\Hom(P_i^\vee,p_1^*F^\vee\otimes p_2^*A),
$$
hence $P_i^\vee \in {}^\perp({}_X\CA_j)$ for all $j \ne i$.
Therefore $P_i^\vee \in {}_X\CB_i$ by~\eqref{bi}, so $P_i \in {}_X\CB_i^\vee$.
\end{proof}

\begin{corollary}\label{so-t}
Let $\D^b(X) = \lan \CA_1,\dots,\CA_m\ran$ be a semiorthogonal decomposition and
let $P_i$ be the kernels of the corresponding projection functors onto $\CA_i$.
Then we have $\HH^\bullet(X\times X,P_i\otimes P_j^\tra) = 0$ for $i \ne j$.
\end{corollary}
\begin{proof}
By Proposition~\ref{ki} we have $P_i \in \CA_i\boxtimes\CB_i^\vee$, hence $P_j^\tra \in \CB_j^\vee\boxtimes\CA_j$,
so it suffices to note that $\HH^\bullet(X,\CA_i\otimes\CB_j^\vee) = \Hom(\CB_j,\CA_i) = 0$ for $i \ne j$
by the definition~\eqref{bi} of $\CB_j$.
\end{proof}

\begin{corollary}\label{so}
Let $\D^b(X) = \lan \CA_1,\dots,\CA_m\ran$ be a semiorthogonal decomposition and
let $P_i$ be the kernels of the corresponding projection functors onto $\CA_i$.
Then we have $\Hom^\bullet(P_i,P_j\circ S_X) = 0$ for $i \ne j$.
\end{corollary}
\begin{proof}
By Proposition~\ref{ki} we have $P_i \in \CA_i\boxtimes\CB_i^\vee$,
hence $P_j\circ S_X \in \CA_j\boxtimes(\CB_j^\vee\otimes\omega_X)$.
So, for $i > j$ the required vanishing follows from the semiorthogonality of $\CA_i$ and $\CA_j$ and
for $i < j$ it follows from the semiorthogonality of $\CB_i^\vee$ and $\CB_j^\vee\otimes\omega_X$
(indeed, by Serre duality we have
$\Hom(\CB_i^\vee,\CB_j^\vee\otimes\omega_X[\dim X]) =
\Hom(\CB_j^\vee,\CB_i^\vee)^\vee =
\Hom(\CB_i,\CB_j)^\vee$
which is zero since $\D^b(X) = \langle \CB_m,\dots,\CB_1 \rangle$
is a semiorthogonal decomposition).
%
%
\end{proof}

\begin{corollary}\label{so-new}
Let $\D^b(X) = \lan \CA_1,\dots,\CA_m\ran$ be a semiorthogonal decomposition and
let $P_i$ be the kernels of the corresponding projection functors.
Then we have $P_i\circ P_j^* = 0$ for $i < j$ and $P_i\circ P_j^! = 0$ for $i > j$.
\end{corollary}
\begin{proof}
We have $P_i \in \CA_i\boxtimes\CB_i^\vee$ by Proposition~\ref{ki}, hence
$P_j^* \in \CB_j\boxtimes (\CA_j^\vee\otimes\omega_X)$, $P_j^! \in (\CB_j\otimes\omega_X)\boxtimes \CA_j^\vee$
by Lemma~\ref{adjs}. So, the first vanishing follows from semiorthogonality
of $(\CB_j,\CB_i)$ and the second --- from the semiorthogonality of $(\CB_i,\CB_j)$ and the Serre duality.
\end{proof}

\begin{corollary}\label{so1}
Let $\D^b(X) = \lan \CA_1,\dots,\CA_i,\CA_{i+1},\dots,\CA_m\ran$ be a semiorthogonal decomposition.
Let $\D^b(X) = \lan \CA_1,\dots,\CA'_{i+1},\CA_i,\dots,\CA_m\ran$ be the semiorthogonal decomposition
obtained by the left mutation of $\CA_{i+1}$ through $\CA_i$. Let $P_i$ and $P'_i$ be the kernels of the
projection functors onto $\CA_i$ with respect to these decompositions.
Then we have canonical isomorphisms
$\Hom^\bullet(P_i,P_i) \cong \Hom^\bullet(P'_i,P_i) \cong \Hom^\bullet(P'_i,P'_i)$.
\end{corollary}
\begin{proof}
Let $\CA = \langle \CA_i,\CA_{i+1} \rangle = \langle \CA'_{i+1}, \CA_i \rangle$
and let $D$ be the component of $\Delta_*\CO_X$ in $\CA_X$.
Then we have distinguished triangles
$$
P_{i+1} \to D \to P_i,\qquad\text{and}\qquad
P'_i \to D \to P'_{i+1}
$$
Consider the composition map $P'_i \to D \to P_i$ and extend it to a distinguished triangle
\begin{equation}\label{kkp}
K \to P'_i \to P_i.
\end{equation}
By the octahedron axiom we also have a distinguished triangle
\begin{equation}\label{kkp1}
K \to P_{i+1} \to P'_{i+1}.
\end{equation}
Let $\D^b(X) = \langle \CB_m,\dots,\CB_{i+1},\CB_i,\dots,\CB_1 \rangle$
and $\D^b(X) = \langle \CB_m,\dots,\CB'_i,\CB_{i+1},\dots,\CB_1 \rangle$
be the dual decompositions (note that by Lemma~\ref{dual-mut}
the second is obtained from the first by the left mutation of $\CB_i$ through $\CB_{i+1}$).
By Proposition~\ref{ki} we have
$$
P_i \in \CA_{i}\boxtimes \CB_i^\vee,\quad
P'_i \in \CA_{i}\boxtimes (\CB'_i)^\vee,\quad
P_{i+1} \in \CA_{i+1}\boxtimes \CB_{i+1}^\vee,\quad
P'_{i+1} \in \CA'_{i+1}\boxtimes \CB_{i+1}^\vee.
$$
%
%
Triangle~\eqref{kkp} implies that $K \in \CA_{iX}$, while~\eqref{kkp1} implies that $K \in {}_X\CB_{i+1}^\vee$,
which means that
$$
K \in \CA_i\boxtimes\CB_{i+1}^\vee.
$$
Since pairs
$(\CB_i^\vee,\CB_{i+1}^\vee)$ and $(\CB_{i+1}^\vee,{\CB'_i}^\vee)$ are semiorthogonal, we conclude that
$$
\Hom^\bullet(K,P_i) = \Hom^\bullet(P'_i,K) = 0.
$$
Now applying the functor $\Hom(P'_i,-)$ to~\eqref{kkp} we deduce that $\Hom^\bullet(P'_i,P_i) \cong \Hom^\bullet(P'_i,P'_i)$,
and applying the functor $\Hom(-,P_i)$ we deduce that $\Hom^\bullet(P_i,P_i) \cong \Hom^\bullet(P'_i,P_i)$.
\end{proof}

\section{Hochschild homology and cohomology of DG-algebras}\label{sec-dg}

\subsection{Derived categories of DG-algebras}
We refer to~\cite{Ke} for a comprehensive introduction in DG-algebras.
Below we remind the basic definitions.




Recall that a DG-algebra over a field $\kk$ is a complex $C^\bullet$ of $\kk$-modules
with an associative multiplication map $C^\bullet\otimes_\kk C^\bullet \to C^\bullet$
such that for the differential of $C^\bullet$ the (graded) Leibnitz rule holds.
A right DG-module over a DG-algebra $C^\bullet$ is a complex $M^\bullet$ of $\kk$-modules
with an associative multiplication map $M^\bullet\otimes_\kk C^\bullet \to M^\bullet$
such that for the differential of $M^\bullet$ the (graded) Leibnitz rule holds.
A left DG-module is defined analogously, and a DG-bimodule over $C^\bullet$
is a complex of $\kk$-modules with commuting structures of a left and of
a right DG-module.

The derived category $\D(C)$ of a DG-algebra $C$ is defined as the localization of the homotopy category
of all (right) DG-modules over $C$ with respect to the class of quasiisomorphisms. This category is triangulated.
The category $\D^\perf(C)$ of perfect DG-modules over $C$ is the minimal closed under direct summands
triangulated subcategory of $\D(C)$ containing the free DG-module $C$. There is a similar notion
of a DG-category (which can be thought of as a DG-algebra with several objects) and of its derived category.

Triangulated categories equivalent to derived categories of DG-categories
(such triangulated categories are called {\sf enhanced}, see~\cite{BK-E}, and a choice of such DG-category
and of an equivalence is called an {\sf enhancement}) have many useful properties. In particular for any two
objects $M,N$ of an enhanced triangulated category there is a complex $\RHom^\bullet(M,N)$ such that
its $i$-th cohomology is isomorphic to $\Hom(M,N[i])$ and the multiplication of $\Hom$'s lifts to
a multiplication of $\RHom$ complexes. In particular, for any object $M$ of an enhanced triangulated
category $\CT$ there is a DG-algebra $\RHom^\bullet(M,M)$.

The derived category of quasicoherent sheaves on an algebraic variety is well known to be enhanced.
The standard enhancement for it is provided by the DG-category of h-injective complexes of quasicoherent
sheaves (see e.g.~\cite{KSch}, Theorem 14.3.1). Starting from this enhancement one can produce many others.
The most convenient come from appropriate generators.

Recall that an object $\CE$ is a {\sf generator} of a triangulated category $\CT$ if $\CE^\perp = 0$.
An object $\CE$ is {\sf compact}\/ if the functor $\Hom(\CE,-)$ commutes with arbitrary (infinite) coproducts.
All compact objects in $\CT$ form a triangulated subcategory denoted by $\CT_c$.
An object $\CE$ is a {\sf strong generator}\/ for $\CT$ (see~\cite{BV}) if there is an integer $N$
such that any object of $\CT$ can be obtained from $\CE$ by taking finite direct sums, shifts, direct summands
and no more than $N$ cones of morphisms. It is easy to see that any strong generator of $\D^b(X)$
is a generator of the unbounded derived category of quasicoherent sheaves on $X$ 
(indeed, if $\CE$ is a strong generator of $\D^b(X)$ then $\CE^\perp = (\D^b(X))^\perp = 0$).

\begin{theorem}[\cite{Ke-DG}]\label{ke-r}
Assume that $\CT$ is an enhanced triangulated category and $\CE$ is a compact generator of $\CT$.
Let $C = \RHom^\bullet(\CE,\CE)$. Then $\CT \cong \D(C)$ and $\CT_c \cong \D^\perf(C)$.
\end{theorem}

\begin{theorem}[\cite{BV}]\label{bv-r}
Let $X$ be a smooth projective variety. Then the bounded derived category of coherent sheaves on $X$
has a strong generator, which is a compact generator of the unbounded derived category of
quasicoherent sheaves on $X$.
\end{theorem}

The following evident argument shows that any admissible subcategory of $\D^b(X)$ also has a strong generator.


\begin{lemma}
Let $\CE$ be a strong generator of a triangulated category $\CT$ and $\CA \subset \CT$ an admissible subcategory.
Then the component $\CE_\CA$ of $\CE$ in $\CA$ is a strong generator of $\CA$. In particular any admissible
subcategory $\CA \subset \D^b(X)$ is strongly generated. Moreover, there is an equivalence $\CA \cong \D^\perf(C)$
for the DG-algebra $C^\bullet = \RHom^\bullet(\CE_\CA,\CE_\CA)$.
\end{lemma}
\begin{proof}
Denote by $\langle \CE \rangle_k$ the subcategory of $\CT$ consisting of objects which can
be obtained from $\CE$ by taking finite direct sums, shifts, direct summands and no more than $k$ cones of morphisms.
Then it is easy to see that for any exact functor $\Phi:\CT \to \CT'$ we have
$\Phi(\langle \CE \rangle_k) \subset \langle \Phi(\CE) \rangle_k$.
In particular, if $\alpha:\CT \to \CA$ is the projection functor then $\alpha(\langle \CE \rangle_k) \subset \langle \CE_\CA \rangle_k$.
Now if $\CE$ is a strong generator for $\CT$ then $\langle \CE \rangle_N = \CT$ for some $N \in \ZZ$,
hence $\alpha(\CT) \subset \langle \CE_\CA \rangle_N$, but $\CA = \alpha(\CA) \subset \alpha(\CT)$,
hence $\CE_\CA$ is a strong generator for $\CA$.

By Theorem~\ref{bv-r} for any smooth projective variety $X$ the category $\D^b(X)$ has a strong generator.
Hence any admissible subcategory $\CA \subset \D^b(X)$ is strongly generated. Therefore, by~\cite{Ke-DG}
there is an equivalence $\CA \cong \D^\perf(C)$.
\end{proof}

\subsection{Hochschild cohomology and homology}

Each DG-algebra $C^\bullet$ has a canonical structure of a DG-bimodule over itself.
This is used to define the Hochschild cohomology and homology:
$$
\HOH^\bullet(C) = \RHom_{C\otimes C^{\opp}}(C,C),
\qquad
\HOH_\bullet(C) = C\mathop{\mathop{\otimes}^{\mathbf{L}}}\nolimits_{C\otimes C^{\opp}}C.
$$
The Hochschild homology and cohomology of DG algebras is known to be invariant under derived
Morita equivalence (see e.g.~\cite{Ke}).
This allows to define the Hochschild cohomology and homology of an admissible subcategory
of $\D^b(X)$ as those of the DG-algebra of endomorphisms of its appropriate generator.

\begin{definition}
Let $\CA \subset \D^b(X)$ be an admissible subcategory of $\D^b(X)$ for a smooth projective variety $X$.
Let $\CE_\CA$ be a strong generator of $\CA$ and $C_\CA = \RHom^\bullet(\CE_\CA,\CE_\CA)$. Then we define
$$
\HOH^\bullet(\CA) := \HOH^\bullet(C_\CA),\qquad
\HOH_\bullet(\CA) := \HOH_\bullet(C_\CA).
$$
\end{definition}


It turns out, however, that for the computation and investigation
of the Hochschild cohomology and homology of an admissible subcategory
one can forget everything about DG-algebras and to remain within the realm
of algebraic geometry.

\begin{theorem}\label{dghh}
Let $\D^b(X) = \langle \CA_1,\dots,\CA_m \rangle$ be a semiorthogonal decomposition.
Let $P_i \in \D^b(X\times X)$ be the kernels of the corresponding projection functors
onto $\CA_i$. Then
$$
\HOH^\bullet(\CA_i) \cong \Hom^\bullet(P_i,P_i),
\qquad
\HOH_\bullet(\CA_i) = \HH^\bullet(X\times X,P_i\otimes P_i^\tra).
$$
\end{theorem}
\begin{proof}
Let $\CE \in \D^b(X)$ be a strong generator
and $\CE_i$ the component of $\CE$ in $\CA_i$.
Then $\CE_i$ is a strong generator for $\CA_i$.
Moreover, $\CF_i = \BD_i(\CE_i)^\vee$ is a strong generator for $\CB_i^\vee$,
where $\CB_i$ is the component of the dual semiorthogonal decomposition and
$\BD_i:\CA_i \to \CB_i$ is the equivalence given by the iterated mutation functor,
see Section~\ref{ss-mut} for details.

Let
$$
C = \RHom(\CE_i,\CE_i) \cong \RHom(\BD_i(\CE_i),\BD_i(\CE_i)) \cong \RHom(\CF_i,\CF_i)^\opp.
$$
Note that the functors
$$
\epsilon_i:\D^\perf(C) \to \D^b(X),\quad
M \mapsto M\otimes_{C}\CE_i,
\qquad
\phi_i:\D^\perf(C^\opp) \to \D^b(X),\quad
N \mapsto N\otimes_{C^\opp}\CF_i,
$$
are fully faithful with $\CA_i$ and $\CB_i^\vee$ being the essential images (see e.g.~\cite{BV}).
It follows that the functor
$$
\mu_i:\D^\perf(C\otimes C^\opp) \to \D^b(X\times X),\qquad
M \mapsto M \otimes_{C\otimes C^\opp} (\CE_i\boxtimes\CF_i)
$$
is fully faithful as well. Note that $\mu_i$ takes the diagonal bimodule $C$ to
$$
C\otimes_{C\otimes C^\opp} (\CE_i\boxtimes\CF_i) \cong \CE_i\otimes_{C}\CF_i
$$
which is precisely the kernel $P_i$ of the projection functor onto $\CA_i$.
We conclude that
$$
\HOH^\bullet(\CA_i) \cong
\Hom^\bullet_{C\otimes C^\opp}(C,C) \cong
\Hom^\bullet(P_i,P_i),
$$
which is the first claim of the Theorem.

Further, note that for any $M \in \D^\perf(C)$, $N \in \D^\perf(C^\opp)$ we have
$$
M\otimes_{C}N \cong \HH^\bullet(X,M\otimes_{C}\CE_i\otimes\CF_i\otimes_{C}N).
$$
Indeed, in case $M = N = C$ the right-hand-side gives
$$
\HH^\bullet(X,\CE_i\otimes\CF_i) \cong
\HH^\bullet(X,\CE_i \otimes \BD_i(\CE_i)^\vee) \cong
\RHom^\bullet(\BD_i(\CE_i),\CE_i) \cong
\RHom^\bullet(\CE_i,\CE_i) =
C
$$
(the third isomorphism is given by Lemma~\ref{bdi})
and for general $M,N$ the formula follows by devissage.

Similarly, for any $K_1,K_2 \in \D^\perf(C\otimes C^\opp)$ we have
$$
K_1\otimes_{C\otimes C^\opp} K_2 \cong \HH^\bullet(X\times X,\mu_i(K_1)\otimes\mu_i(K_2)^\tra).
$$
Indeed, for $K_1 = K_2 = C\otimes C$ we have $\mu_i(C\otimes C) = \CE_i\boxtimes\CF_i$,
so the right-hand-side is isomorphic to
\begin{multline*}
\HH^\bullet(X\times X,(\CE_i\boxtimes\CF_i)\otimes(\CE_i\boxtimes\CF_i)^\tra) =
\HH^\bullet(X\times X,(\CE_i\otimes\CF_i)\boxtimes(\CE_i\boxtimes\CF_i)) = \\ =
\HH^\bullet(X,\CE_i\otimes\CF_i)\otimes \HH^\bullet(X,\CE_i\boxtimes\CF_i) =
C\otimes C
\end{multline*}
which coincides with the left-hand-side. For general $K_1,K_2$ we conclude by devissage.

Finally, applying the above formula for $K_1 = K_2 = C$ we obtain an isomorphism
$$
\HOH_\bullet(\CA_i) \cong
C\otimes_{C\otimes C^\opp}C \cong
\HH^\bullet(X\times X,\mu_i(C)\otimes\mu_i(C)^\tra) \cong
\HH^\bullet(X\times X,P_i\otimes P_i^\tra),
$$
which is the second claim of the Theorem.
\end{proof}

Actually, the formula for the Hochschild homology given by the above Theorem is not very convenient.
We will need the following its reinterpretation.

\begin{proposition}\label{hhserre}
Let $\D^b(X) = \langle \CA_1,\dots,\CA_m \rangle$ be a semiorthogonal decomposition.
Let $P_i \in \D^b(X\times X)$ be the kernels of the corresponding projection functors
onto $\CA_i$. Then $\HOH_\bullet(\CA_i) \cong \Hom(P_i,P_i\circ S_X)$.
\end{proposition}
\begin{proof}
By the above Theorem we have $\HOH_\bullet(\CA) \cong \HH^\bullet(X\times X,P_i\otimes P_i^\tra)$.
Consider the isomorphism $\HH^\bullet(X\times X,\Delta_*\CO_X\otimes\Delta_*\CO_X) \cong \Hom(S_X^{-1},\Delta_*\CO_X)$
of Lemma~\ref{hohserre}. The filtration on $\Delta_*\CO_X$ induced by the semiorthogonal decomposition induces
filtrations on both sides of this isomorphisms. Since the isomorphism is functorial we conclude that it gives
isomorphisms on the factors of this filtrations, i.e. isomorphisms
$$
\HH^\bullet(X\times X,\Delta_*\CO_X\otimes P_i) \cong \Hom(S_X^{-1},P_i).
$$
Now consider the filtration on $\Delta_*\CO_X$ with the factors $P_j$.
It gives a filtration on $\Hom(S_X^{-1},P_i) \cong \Hom(\Delta_*\CO_X,P_i\circ S_X)$
with factors~$\Hom(P_j,P_i\circ S_X)$. By Corollary~\ref{so} all factors with $i \ne j$ vanish.

Similarly, transposing the filtration we obtain a filtration on
$\Delta_*\CO_X \cong (\Delta_*\CO_X)^\tra$ which gives a filtration on the left-hand-side
with factors $\HH^\bullet(X\times X,P_i\otimes P_j^\tra)$.
By Corollary~\ref{so-t} all factors with $i \ne j$ vanish.
So we conclude that $\HH^\bullet(X\times X,P_i\otimes P_i^\tra) \cong \Hom(P_i,P_i\circ S_X)$.
\end{proof}

\section{Generalized Hochschild cohomology}\label{ghoh}

Combining the results of Theorem~\ref{dghh} and Proposition~\ref{hhserre}
we see that both the Hochschild cohomology and homology of an admissible subcategory
can be written as $\Hom(P,P\circ T)$, where $P$ is the kernel of the projection functor
to $\CA$ and $T = \Delta_*\CO_X$ for cohomology or $T = S_X$ for homology. So, it is reasonable
to consider the graded vector spaces $\Hom(E,E\circ T)$ for arbitrary $E,T \in \D^b(X\times X)$
and investigate their behavior with respect to $E$ and $T$.



%

\begin{definition}
The {\sf Hochschild cohomology of $X$ with support in $T$ and coefficients in $E$}\/ is defined as
$$
\HOH_T^\bullet(X,E) = \Hom^\bullet(E,E\circ T)
$$
for any kernels $E,T \in \D^b(X\times X)$.
\end{definition}

In case of $T = \Delta_*\CO_X$ we write $\HOH^\bullet(X,E) = \HOH^\bullet_{\Delta_*\CO_X}(X,E)$
and call it the {\sf Hochschild cohomology of $X$ with coefficients in $E$}. Similarly,
in case of $T = S_X$ we write $\HOH_\bullet(X,E) = \HOH^\bullet_{S_X}(X,E)$
and call it the {\sf Hochschild homology of $X$ with coefficients in $E$}.


Fix a support $T \in \D^b(X\times X)$ with respect to which the Hochschild cohomology will be considered.
Let $K,E \in \D^b(X\times X)$. Consider an element $h \in \HOH^\bullet_T(X,E) = \Hom(E,E\circ T)$.
Taking the convolution with $K$ we obtain a map $K\circ E \xrightarrow{\ a_{K,E,T}(\id_K\circ h)\ } K\circ E\circ T$.
This defines a graded map
$$
\gamma_{K}:\HOH^\bullet_T(X,E) \to \HOH^\bullet_T(X,K\circ E).
$$
In particular, for $E = \Delta_*\CO_X$ we obtain a map $\gamma_K:\HOH^\bullet_T(X) \to \HOH^\bullet_T(X,K)$.

\begin{lemma}
We have $\gamma_K\circ\gamma_{K'} = \gamma_{K\circ K'}$.
\end{lemma}
\begin{proof}
Follows from the functoriality of the associativity isomorphism
and the pentahedron equation.
\end{proof}

\begin{lemma}\label{comm}
For any morphism of kernels $\phi:K \to L$ and any generalized Hochschild cohomology class $h \in \HOH^\bullet_T(X)$ the following
diagram commutes
$$
\xymatrix{
K \ar[rr]^{\phi} \ar[d]_{\gamma_K(h)} && L \ar[d]^{\gamma_L(h)} \\
K\circ T \ar[rr]^{\phi\circ\id_{T}} && L\circ T
}
$$
\end{lemma}
\begin{proof}
Indeed, by definition $\gamma_K(h) = \id_K\circ h$, $\gamma_L(h) = \id_L\circ h$,
so it remains to note that
$$
(\phi\circ\id_{T})(1\circ h) = \phi\circ h = (1\circ h)(\phi\circ\id_{\Delta_*\CO_X})
$$
and that $\phi\circ\id_{\Delta_*\CO_X} = \phi$.
\end{proof}

The generalized Hochschild cohomology $\HOH^\bullet_T$ is exact with respect to $T$.

\begin{lemma}
If $T_1 \to T_2 \to T_3$ is a distinguished triangle then
there is a distinguished triangle of generalized Hochschild cohomology
$$
\HOH^\bullet_{T_1}(X,K) \to \HOH^\bullet_{T_2}(X,K) \to \HOH^\bullet_{T_3}(X,K).
$$
\end{lemma}
\begin{proof}
Evident.
\end{proof}

Sometimes, the generalized Hochschild cohomology $\HOH^\bullet_T(X,K)$ is also linear
with respect to $K$, but for this some additional conditions are required.

\begin{proposition}\label{add}
Let $K_1 \xrightarrow{\phi_1} K \xrightarrow{\phi_2} K_2$ be an exact triangle and assume that
\begin{equation}\label{vancond1}
\Hom^\bullet(K_1,K_2\circ T) = 0.
\end{equation}
Then there is a map $\pi:\HOH^\bullet_T(X,K) \to \HOH^\bullet_T(X,K_1)\oplus\HOH^\bullet_T(X,K_2)$
such that the following diagram commutes
\begin{equation}\label{pig}
\vcenter{\xymatrix{
\HOH^\bullet_T(X) \ar[rrr]^-\Delta \ar[d]_{\gamma_K} &&& \HOH^\bullet_T(X)\oplus\HOH^\bullet_T(X) \ar[d]^{\gamma_{K_1}\oplus\gamma_{K_2}} \\
\HOH^\bullet_T(X,K) \ar[rrr]^-\pi &&& \HOH^\bullet_T(X,K_1)\oplus\HOH^\bullet_T(X,K_2)
}}
\end{equation}
Moreover, if additionally we have
\begin{equation}\label{vancond2}
\Hom^\bullet(K_2,K_1\circ T) = 0
\end{equation}
then the map $\pi:\HOH^\bullet_T(X,K) \to \HOH^\bullet_T(X,K_1)\oplus\HOH^\bullet_T(X,K_2)$ is an isomorphism.
In particular, for any $h \in \HOH^\bullet_T(X)$ we have
$\gamma_K(h) = 0$ if and only if $\gamma_{K_1}(h) = \gamma_{K_2}(h) = 0$.
\end{proposition}
\begin{proof}
We have the following distinguished triangles in $\D^b(X\times X)$
\begin{equation}\label{hhdiagr}
\vcenter{\xymatrix@R=10pt{
K_1 \ar[rr]^{\phi_1} && K  \ar[rr]^{\phi_2} && K_2,  \\
K_1\circ T \ar[rr]^{\phi_1\circ\id_{T}} && K\circ T \ar[rr]^{\phi_2\circ\id_{T}} && K_2\circ T.
}}
\end{equation}
Consider the following commutative squares
\begin{equation}\label{iotapi}
\vcenter{\xymatrix{
\Hom(K_2,K\circ T) \ar[r]^-{\phi_2^*} \ar[d]_-{\phi_{2*}} &
\Hom(K,K\circ T) \ar[d]_-{\phi_{2*}} \\
\Hom(K_2,K_2\circ T) \ar[r]^-{\phi_2^*} &
\Hom(K,K_2\circ T)
}}
\qquad
\vcenter{\xymatrix{
\Hom(K,K_1\circ T) \ar[r]^-{\phi_{1*}} \ar[d]_-{\phi_1^*} &
\Hom(K,K\circ T) \ar[d]_-{\phi_1^*} \\
\Hom(K_1,K_1\circ T) \ar[r]^-{\phi_{1*}} &
\Hom(K_1,K\circ T)
}}
\end{equation}
It follows immediately from~\eqref{vancond1} that
the bottom arrows in these squares are isomorphisms.
Therefore, the composition of the right arrows in these squares
with the inverses of the bottom arrows give maps
$\pi_{\phi_2}:\HOH^\bullet_T(X,K) \to \HOH^\bullet_T(X,K_2)$ and
$\pi_{\phi_1}:\HOH^\bullet_T(X,K) \to \HOH^\bullet_T(X,K_1)$.
Moreover, it follows immediately from Lemma~\ref{comm} that
for any $h \in \HOH^\bullet_T(X)$ we have $\phi_2^*(\gamma_{K_2}(h)) = \phi_{2*}(\gamma_K(h))$,
so $\gamma_{K_2}(h) = \pi_{\phi_2}(\gamma_K(h))$.
Analogously, $\gamma_{K_1}(h) = \pi_{\phi_1}(\gamma_K(h))$, so diagram~\eqref{pig} commutes.

Further, if~\eqref{vancond2} is also true then the left arrows in squares~\eqref{iotapi} are also isomorphisms.
Therefore, the composition of their inverses with the top arrows give morphisms
$\iota_{\phi_2}:\HOH^\bullet_T(X,K_2) \to \HOH^\bullet_T(X,K)$ and
$\iota_{\phi_1}:\HOH^\bullet_T(X,K_1) \to \HOH^\bullet_T(X,K)$.
Moreover, commutativity of squares~\eqref{iotapi}
implies that we have equalities $\pi_{\phi_2}\circ\iota_{\phi_2} = \id_{\HOH_T^\bullet(X,K_2)}$,
$\pi_{\phi_1}\circ\iota_{\phi_1} = \id_{\HOH_T^\bullet(X,K_1)}$.
In particular, $\iota_{\phi_1}$ and $\iota_{\phi_2}$ are embeddings while
$\pi_{\phi_1}$ and $\pi_{\phi_2}$ are surjections.

Let us show that the map $(\pi_{\phi_1},\pi_{\phi_2}):\HOH^\bullet_T(X,K) \to \HOH^\bullet_T(X,K_1)\oplus\HOH^\bullet_T(X,K_2)$ is an isomorphism
if~\eqref{vancond2} is satisfied.
Indeed, applying the functor $\Hom(-,K\circ T[k])$ to the upper line of the diagram~\eqref{hhdiagr}
and recalling the definition of $\iota_{\phi_2}$ and $\pi_{\phi_1}$ we see that there is a distinguished triangle
$$
\xymatrix@1{\HOH^\bullet_T(X,K_2) \ar[rr]^-{\iota_{\phi_2}} && \HOH^\bullet_T(X,K) \ar[rr]^-{\pi_{\phi_1}} && \HOH^\bullet_T(X,K_1)}
$$
But since $\pi_{\phi_2}\iota_{\phi_2} = \id$, as we already proved,
this triangle splits and $(\pi_{\phi_1},\pi_{\phi_2})$ give an isomorphism.
\end{proof}

\section{Functoriality of generalized Hochschild cohomology}\label{ghohfun}

In this Section we construct action of functors on generalized Hochschild cohomology.

\begin{definition}
Assume given support kernels $T_X \in \D^b(X\times X)$, $T_Y \in \D^b(Y\times Y)$.
We will say that a kernel $K \in \D^b(X\times Y)$ is {\sf S-intertwining between
$T_X$ and $T_Y$ of degree $d$}\/ if there is given an isomorphism
$$
(T_Y^!\circ S_Y)\circ K \cong K\circ (T_X^!\circ S_X)[d].
$$
\end{definition}

\begin{lemma}
Any kernel $K \in \D^b(X\times Y)$ is S-intertwining of degree $0$ between $T_X = S_X$ and $T_Y = S_Y$.
\end{lemma}
\begin{proof}
Indeed, $S_X^! = S_X^{-1}$ by Lemma~\ref{adj-diag}, so $S_X^!\circ S_X \cong \Delta_*\CO_X$ by Lemma~\ref{conv-diag}.
\end{proof}

Alternatively, this can be deduced from Corollary~\ref{sks}
using the following

\begin{lemma}
A kernel $K \in \D^b(X\times Y)$ is S-intertwining between $T_X$ and $T_Y$ of degree $d$
if and only if there is an isomorphism
\begin{equation}\label{tkcomp}
T_X \circ K^* \cong K^!\circ T_Y[d].
\end{equation}
\end{lemma}
\begin{proof}
Note that $K^! \cong S_X\circ K^*\circ S_Y^{-1}$. Therefore~\eqref{tkcomp}
is equivalent to $T_X\circ K^* \cong S_X\circ K^*\circ S_Y^{-1}\circ T_Y[d]$.
Taking the left convolution with $S_X^{-1}$ and considering the right adjoints,
we see that this is equivalent to the definition of being $S$-intertwining.
\end{proof}

\begin{corollary}
If a kernel $K\in\D^b(X\times Y)$ gives an equivalence of $\D^b(X)$ and $\D^b(Y)$
then $K$ is S-intertwining between $T_X = \Delta_*\CO_X$ and $T_Y = \Delta_*\CO_Y$.
\end{corollary}
\begin{proof}
Indeed, if $K$ gives an equivalence then
$\Delta_*\CO_X \circ K^* \cong K^* \cong K^! \cong K^!\circ \Delta_*\CO_Y$.
\end{proof}

\begin{lemma}\label{phik}
Let $X$, $Y$ be smooth projective varieties and $T_X \in \D^b(X\times X)$,
$T_Y \in \D^b(Y\times Y)$ support kernels. If $\Phi_K:\D^b(X) \to \D^b(Y)$ is a functor
given by an S-intertwining between $T_X$ and $T_Y$ of degree~$d$ kernel $K \in \D^b(X\times Y)$
then there exists a map
$\phi_K:\HOH_{T_X}^\bullet(X,\Delta_*\CO_X) \to \HOH_{T_Y}^{\bullet}(Y,\Delta_*\CO_Y)$
of degree~$d$ such that $\phi_{K\circ L} = \phi_K \circ \phi_L$.
\end{lemma}
\begin{proof}
For any $h \in \HOH^\bullet_{T_X}(X,\Delta_*\CO_X) = \Hom(\Delta_*\CO_X,T_X)$ consider the composition
$$
\Delta_*\CO_Y \xrightarrow{\ \lambda_K\ }
K\circ K^* \cong
K \circ \Delta_*\CO_X \circ K^* \xrightarrow{\ h\ }
K \circ T_X \circ K^* \cong
K \circ K^! \circ T_Y[d] \xrightarrow{\ \rho_K\ }
T_Y[d],
$$
where the maps $\lambda_K$ and $\rho_K$ are defined in Lemma~\ref{adjs}.
The composition can be considered as an element of $\HOH^\bullet_{T_Y}(Y,\Delta_*\CO_Y)$ which we denote by $\phi_K(h)$.
This way a map $\phi_K:\HOH^\bullet_{T_X}(X,\Delta_*\CO_X) \to \HOH^\bullet_{T_Y}(Y,\Delta_*\CO_Y)$ is defined.
The required functoriality can be checked straightforwardly.
\end{proof}

In particular, we see that any kernel $K$ gives a map on Hochschild homology
and any equivalence gives a map on Hochschild cohomology. We would also need
a condition ensuring that a kernel $K$ gives a map on generalized Hochschild
cohomology.

\begin{definition}
Assume given coefficients $E_X \in \D^b(X\times X)$, $E_Y \in \D^b(Y\times Y)$  for
Hochschild cohomology. We will say that a kernel $K \in \D^b(X\times Y)$ is
{\sf intertwining between $E_X$ and $E_Y$ of degree $d$}\/ if there is given an isomorphism
$$
E_Y \circ K \cong K\circ E_X[d].
$$
\end{definition}

\begin{lemma}\label{phik1}
Let $X$, $Y$ be smooth projective varieties,
and $\Phi_K:\D^b(X) \to \D^b(Y)$ a functor given by a kernel $K \in \D^b(X\times Y)$ which
is S-intertwining between $T_X$ and $T_Y$ of degree $d_T$ and intertwining between $E_X$ and $E_Y$ of degree $d_E$.
Then there exists a map $\phi_K:\HOH_{T_X}^\bullet(X,E_X) \to \HOH_{T_Y}^{\bullet+d_T+d_E}(Y,E_Y)$ such that
the diagram
$$
\xymatrix{
\HOH^\bullet_{T_X}(X) \ar[rr]^-{\phi_K} \ar[d]_{\gamma_{E_X}} && \HOH^\bullet_{T_Y}(Y) \ar[d]^{\gamma_{E_Y}} \\
\HOH^\bullet_{T_X}(X,E_X) \ar[rr]^-{\phi_K} && \HOH^\bullet_{T_Y}(Y,E_Y)
}
$$
is commutative. Moreover, if $L \in \D^b(Z\times Y)$ S-intertwines between $T_Y$ and $T_Z$ and
intertwines between $E_Y$ and $E_Z$ then $\phi_{K\circ L} = \phi_K \circ \phi_L$.
\end{lemma}
\begin{proof}
For any $h \in \HOH^\bullet_{T_X}(X,E_X) = \Hom(E_X,E_X\circ T_X)$ consider the composition
\begin{multline*}
E_Y \xrightarrow{\ \lambda_K\ }
E_Y\circ K\circ K^* \cong
K\circ E_X\circ K^*[d_E] \xrightarrow{\ h\ } \\
K \circ E_X\circ T_X \circ K^*[d_E] \cong
E_Y\circ K \circ K^!\circ T_Y[d_E+d_T] \xrightarrow{\ \rho_K\ }
E_Y\circ T_Y[d_E + d_T].
\end{multline*}
The composition can be considered as an element of $\HOH^\bullet_{T_Y}(Y,E_Y)$ which we denote by $\phi_K(h)$.
This way a map $\phi_K:\HOH^\bullet_{T_X}(X,E_X) \to \HOH^\bullet_{T_Y}(Y,E_Y)$ is defined.
The required functoriality can be checked straightforwardly.

Commutativity of $\phi_K$ with $\gamma$ follows from the commutativity of the following diagram
$$
\newcommand{\scirc}{\!\circ\!}
\xymatrix{
E_Y \ar[r]^-{\lambda_K} \ar@{=}[d] &
E_Y\scirc K\scirc K^* \ar[r]^-h \ar@{=}[d] \ar[dr]^\kappa &
E_Y\scirc K \scirc T_X \scirc K^* \ar[r]^-\tau \ar[dr]^\kappa &
E_Y\scirc K \scirc K^! \scirc T_Y \ar[r]^-{\rho_K} \ar@{=}[dr] &
E_Y\scirc T_Y \ar@{=}[dr] \\
E_Y \ar[r]^-{\lambda_K} &
E_Y\scirc K\scirc K^* \ar[r]^-\kappa &
K\scirc E_X\scirc K^* \ar[r]^-h &
K\scirc E_X \scirc T_X \scirc K^* \ar[r]^-{\kappa^{-1}\scirc\tau} &
E_Y\scirc K \scirc K^! \scirc T_Y \ar[r]^-{\rho_K} &
E_Y\scirc T_Y
}
$$
where $\kappa:E_Y\circ K \to K\circ E_X$ and $\tau:T_X\circ K^* \to K^!\circ T_Y$ stand for the intertwining isomorphisms.
Indeed, in the above diagram the top line gives $\gamma_{E_Y}(\phi_K(h))$ while the bottom line gives $\phi_K(\gamma_{E_X}(h))$.
\end{proof}

\section{Hochschild homology and cohomology of an admissible subcategory}\label{hohadm}

Given an admissible subcategory $\CA \subset \D^b(X)$ we consider any semiorthogonal
decomposition of $\D^b(X)$ containing $\CA$ as a component and consider the kernel
$P \in \D^b(X\times X)$ of the corresponding projection functor onto $\CA \subset \D^b(X)$.
Recall that in Section~\ref{sec-dg} we identified the Hochschild homology and cohomology
of $\CA$ with the vector spaces $\Hom^\bullet(P,P)$ and $\Hom^\bullet(P,P\circ S_X)$.
It follows that whenever we have another admissible subcategory
$\CB \subset \D^b(Y)$ with the kernel of the projection functor $Q \in \D^b(Y,Y)$,
if $\CB$ is equivalent to $\CA$ and the equivalence has a DG-enhancement,
then we have isomorphisms
$$
\Hom^\bullet(P,P) \cong \Hom^\bullet(Q,Q),
\qquad
\Hom^\bullet(P,P\circ S_X) \cong \Hom^\bullet(Q,Q\circ S_Y).
$$
The goal of this Section is to make these isomorphisms explicit.

\subsection{Hochschild cohomology}

Note that a priori the spaces $\Hom^\bullet(P,P)$ and $\Hom^\bullet(P,P\circ S_X)$
depend also on the choice of the projection functor $P$, that is on the choice of a semiorthogonal
decomposition of $\D^b(X)$ containing $\CA$ as a component. We start with analyzing this
dependence. Note that the kernel $P$ does not depend on a subdivision of components
distinct from $\CA$, so we may assume that the first decomposition is
$\D^b(X) = \langle \CA',\CA,\CA'' \rangle$ for some $\CA',\CA'' \subset \D^b(X)$.
On the other hand, consider the decomposition $\D^b(X) = \langle \CA^\perp,\CA \rangle$.
Let $P$ and $P'$ be the kernels of the corresponding projection functors onto $\CA$.

\begin{proposition}\label{mut-hcoh}
There is a morphism $P' \to P$ such that the induced maps
$\Hom^\bullet(P',P') \to \Hom^\bullet(P',P)$ and $\Hom^\bullet(P,P) \to \Hom^\bullet(P',P)$
are isomorphisms.
\end{proposition}
\begin{proof}
The left mutation of $\CA''$ through $\CA$ takes the first decomposition to
$\D^b(X) = \langle \CA',\BL_2(\CA''),\CA \rangle$ which is clearly a subdivision
of the second decomposition. Now the Proposition easily follows from
Corollary~\ref{so1}.
\end{proof}

Now we are ready to construct an explicit isomorphism in general case
for the Hochschild cohomology.
Let $\xi:\CA \to \CB$ be an equivalence and consider the functor
$\Phi:\D^b(X) \to \D^b(Y)$,
$$
\Phi = \beta\circ\xi\circ\alpha^*,
$$
where $\alpha:\CA \to \D^b(X)$ and $\beta:\CB \to \D^b(Y)$ are the embedding functors.
Note that this is a right splitting functor (see~\cite{K-HPD}).
In particular, $\Phi\Phi^! \cong \beta\beta^!$ is the right projection onto $\CB$,
while $\Phi^!\Phi \cong \alpha\alpha^*$ is the left projection onto $\CA$.
The main technical assumption we need is the geometricity of $\Phi$.
Namely, we assume that $\Phi$ is isomorphic to a kernel functor $\Phi_\CE$
given by an appropriate kernel $\CE \in \D^b(Y\times X)$.
Note that this condition is more or less equivalent to an existence
of a DG-enhancement for $\Phi$.

\begin{theorem}\label{extcor}
Let $P$ be the kernel of the left projection functor onto $\CA$ and $Q$ be the kernel
of the right projection functor onto $\CB$. If a right splitting functor $\Phi_\CE:\D^b(X) \to \D^b(Y)$
gives an equivalence of subcategories $\CA \subset \D^b(X)$ and $\CB \subset \D^b(Y)$
then we have isomorphisms
$$
\Hom^\bullet(P,P) \cong \Hom^\bullet(\CE,\CE) \cong \Hom^\bullet(\CE^!,\CE^!) \cong \Hom^\bullet(Q,Q),
$$
where $\CE^!$ is the kernel of the right adjoint functor.
\end{theorem}
\begin{proof}
Consider the functor $\tilde\Phi:\D^b(X\times X) \to \D^b(Y\times X)$
given by the kernel
$$
\tilde{\CE} = p_{24}^*(\Delta_*\CO_X) \otimes p_{13}^*\CE \in \D^b(Y\times X\times X\times X).
$$
%
A simple computation shows that $\tilde\Phi$ takes $\Delta_*\CO_X \in \D^b(X\times X)$ to $\CE \in \D^b(Y\times X)$.
Further, by Theorem~6.4 of~\cite{K-FBC} $\tilde\Phi$ is fully faithful on the subcategory
$\CA_X$ of $\D^b(X\times X)$, vanishes on ${}^\perp\CA_X$, and induces an equivalence
of $\CA_X$ with $\CB_X \subset \D^b(Y\times X)$. It follows that
$$
\tilde\Phi(P) = \tilde\Phi(\Delta_*\CO_X) = \CE,
$$
hence
$$
\Hom^\bullet(P,P) \cong \Hom^\bullet(\CE,\CE).
$$


Similarly, the functor $\tilde\Phi^!:\D^b(Y\times Y) \to \D^b(X\times Y)$
given by the kernel
$$
\tilde\CE^! = p_{24}^*(\Delta_*\CO_X) \otimes p_{13}^*\CE^! \in \D^b(X\times Y\times Y\times Y).
$$
%
%
takes $\Delta_*\CO_Y \in \D^b(Y\times Y)$ to $\CE^! \in \D^b(X\times Y)$.
By Theorem~6.4 of~\cite{K-FBC} this functor is fully faithful on the subcategory
$\CB_Y$ of $\D^b(Y\times Y)$, vanishes on $\CB_Y^\perp$, and induces an equivalence
of $\CB_Y$ with $\CA_Y \subset \D^b(X\times Y)$. It follows that
$$
\tilde\Phi^!(Q) = \tilde\Phi(\Delta_*\CO_Y) = \CE^!,
$$
hence
$$
\Hom^\bullet(Q,Q) \cong \Hom^\bullet(\CE^!,\CE^!).
$$
So, it remains to note that the functor $K \mapsto K^!$ is an antiequivalence of $\D^b(Y\times X)$ with $\D^b(X\times Y)$,
hence induces a canonical isomorphism $\Hom^\bullet(\CE^!,\CE^!) \cong \Hom^\bullet(\CE,\CE)$.
\end{proof}

A combination of Proposition~\ref{mut-hcoh} with Theorem~\ref{extcor} allows to construct the desired isomorphism.

\subsection{Hochschild homology}

Now we consider to the case of Hochschild homology. In this case it turns out that we don't need to investigate
the dependance on the choice of a semiorthogonal decomposition since one can identify $\Hom(P,P\circ S_X)$
with a certain {\em subspace}\/ of $\HOH_\bullet(X)$ depending only on $\CA$.
This easily follows from the following result.

Recall the maps
$\phi_P:\HOH_\bullet(X) \to \HOH_\bullet(X)$ and $\gamma_P:\HOH_\bullet(X) \to \HOH_\bullet(X,P)$
defined in Sections~\ref{ghohfun} and~\ref{ghoh} respectively.

\begin{theorem}\label{addone}
Let $\D^b(X) = \lan \CA_1,\dots,\CA_n)$ be a semiorthogonal decomposition and
let $P_i$ be the kernels of the corresponding projection functors onto $\CA_i$.
Then

\noindent$(i)$
the map
$$
\HOH_\bullet(X) \xrightarrow{\quad(\gamma_{P_1},\dots,\gamma_{P_n})\quad} \HOH_\bullet(X,P_1)\oplus\dots\oplus\HOH_\bullet(X,P_n)
$$
is an isomorphism;

\noindent$(ii)$
the maps $\phi_{P_i} : \HOH_\bullet(X) \to \HOH_\bullet(X)$ are idempotents
summing up to identity, that is
$$
\phi_{P_i}\circ\phi_{P_j} = \delta_{ij}\phi_{P_j},
\qquad
\sum\phi_{P_i} = \id;
$$

\noindent$(iii)$
$\Im\phi_{P_i} = \bigcap_{j\ne i} \Ker \gamma_{P_j}$.
\end{theorem}
\begin{proof}
$(i)$ Let $D_i$ be the truncation functors of the semiorthogonal decomposition.
It is clear that $D_i$ is the projection functor onto the subcategory $\lan\CA_{i+1},\dots,\CA_n\ran \subset \D^b(X)$
and we have distinguished triangles $D_i \to D_{i-1} \to K_i$. Thus the conditions~\eqref{vancond1} and~\eqref{vancond2}
of Proposition~\ref{add} are satisfied, hence the maps $\gamma_{D_i}$, $\gamma_{D_{i-1}}$, $\gamma_{K_i}$ induce
a decomposition $\HOH_\bullet(X,D_{i-1}) = \HOH_\bullet(X,D_i) \oplus \HOH_\bullet(X,K_i)$.
Iterating we obtain a decomposition $\HOH_\bullet(X,D_0) = \oplus \HOH_\bullet(X,K_i)$ induced by $\gamma_{K_i}$.
Since $D_0 = \Delta_*\CO_X$ this is what we need.

$(ii)$
The semiorthogonal decomposition $\D^b(X\times X) = \lan \CA_{1X},\dots,\CA_{nX} \ran$
gives a filtration
$$
0 = D_n \to D_{n-1} \to \dots \to D_1 \to D_0 = \Delta_*\CO_X
$$
the factors of which are $P_i$. It induces two filtrations on
$(\Delta_*\CO_X)^* \cong (\Delta_*\CO_X)^! \cong \Delta_*\CO_X$,
$$
\Delta_*\CO_X = D_0^* \to D_1^* \to \dots \to D_{n-1}^* \to D_n^* = 0,
\quad\text{and}\quad
\Delta_*\CO_X = D_0^! \to D_1^! \to \dots \to D_{n-1}^! \to D_n^! = 0,
$$
with factors being $P_i^*$ and $P_i^!$ respectively.
It follows that the convolutions $(\Delta_*\CO_X)\circ(\Delta_*\CO_X)^* \cong \Delta_*\CO_X$
and $(\Delta_*\CO_X)\circ S_X\circ (\Delta_*\CO_X)^* \cong (\Delta_*\CO_X)\circ(\Delta_*\CO_X)^!\circ S_X \cong S_X$
have bifiltrations with factors $P_i\circ P_j^*$ and $P_i\circ S_X\circ P_j^* \cong P_i\circ P_j^!\circ S_X$ respectively.
Note that $P_i \in \CA_{iX}$, while $P_j^* \in {}_X(\CA_j^\vee\otimes\omega_X)$ by Lemma~\ref{adjs}
and the same is true for $P_j^!\circ S_X$
by Lemma~\ref{adjs} and Lemma~\ref{conv-diag}.
Hence
$$
P_i\circ P_j^*,\
P_i\circ S_X\circ P_j^*,\
P_i\circ P_j^!\circ S_X \in \CA_i\boxtimes (\CA_j^\vee\otimes\omega_X).
$$
Therefore the above bifiltrations are the bifiltrations corresponding to the semiorthogonal decomposition
$$
\D^b(X\times X) = \lan \CA_i\boxtimes (\CA_j^\vee\otimes\omega_X) \ran_{i,j = 1}^n.
$$
By the functoriality of the filtration of a semiorthogonal decomposition we deduce
that every Hochschild homology class $h:\Delta_*\CO_X \to S_X$ is compatible with these filtrations,
and induces on the factors maps
$$
P_i\circ P_j^* \cong P_i\circ\Delta_*\CO_X\circ P_j^* \xrightarrow{\quad \id_{P_i}\circ h\circ\id_{P_j^*}\quad } P_i\circ S_X\circ P_j^* \cong P_i\circ P_j^!\circ S_X.
$$
By Corollary~\ref{so-new} we have $P_i\circ P_j^* = 0$ for $i < j$ and $P_i\circ P_j^! = 0$ for $i > j$.
Therefore the only nontrivial maps on the factors of the filtrations are that for $i = j$,
so the whole map $h:\Delta_*\CO_X \to S_X$ factors as
$$
\Delta_*\CO_X \xrightarrow{\ \oplus \lambda_{P_i}\ }
\bigoplus P_i\circ P_i^* \xrightarrow{\quad \id_{P_i}\circ h\circ\id_{P_j^*}\quad }
\bigoplus P_i\circ S_X\circ P_i^* \cong
\bigoplus P_i\circ P_i^!\circ S_X \xrightarrow{\ \oplus \rho_{P_i}\ }
S_X
$$
which is nothing but the sum of $\phi_{P_i}(h)$. This shows that $\sum \phi_{P_i}(h) = h$, so $\sum \phi_{P_i} = \id$.
The fact that $\phi_{P_i}$ are orthogonal idemptents follows immediately from Lemma~\ref{phik} since
$P_i\circ P_j = 0$ for $i \ne j$ and $P_i\circ P_i \cong P_i$.

$(iii)$
The same argument shows that $\gamma_{P_j}\circ\phi_{P_i} = 0$ for $j \ne i$.
Indeed, $\gamma_{P_j}$ is defined as the left convolution with $P_j$. But $\phi_{P_i}(h)$ factors
through $P_i\circ P_i^*$ and $P_j\circ P_i\circ P_i^* = 0$, so $\gamma_{P_j}\circ\phi_{P_i}$ factors through $0$.
So, it follows that $\Im\phi_{P_i} \subset \bigcap_{j\ne i} \Ker \gamma_{P_j}$, and it remains to check that this
is indeed an equality. But this immediately follows from
$$
\bigoplus_{i=1}^n \Im \phi_{P_i} = \HOH_\bullet(X) = \bigoplus_{i=1}^n \bigcap_{j\ne i} \Ker \gamma_{P_j},
$$
the first equality  follows from $\sum \phi_{P_i} = \id$, and the second follows from $(i)$.
\end{proof}


The claim of the Theorem can be made more precise because of the following

\begin{lemma}\label{phikind}
Take any semiorthogonal decomposition of $\D^b(X)$ containing $\CA$ as a component
and let $P \in \D^b(X\times X)$ be the kernel of the corresponding projection functor.
Then the image of the projector $\phi_P:\HOH_\bullet(X) \to \HOH_\bullet(X)$ does not depend
on the choice of the semiorthogonal decomposition
and the map $\gamma_P:\HOH_\bullet(X) \to \HOH_\bullet(X,P)$ identifies it with $\Hom(P,P\circ S_X)$.
\end{lemma}
\begin{proof}
Let us show that $\Im\phi_P = \lan \Im \phi_E \ran_{E \in \CA_X}$,
the right-hand-side is the linear span in $\HOH_\bullet(X)$ of images of all maps $\phi_E$
for all kernels $E \in \CA_X \subset \D^b(X\times X)$. Since this depends only on $\CA$,
the first part would follow.

So, to check this we include $\CA$ into a semiorthogonal decomposition,
say $\D^b(X) = \lan \CA,{}^\perp\CA \ran$ and let $P$ and $L$ be the kernels of the corresponding
projection functors onto $\CA$ and ${}^\perp\CA$ respectively. Note that for any $E \in \CA_X$
we have $L\circ E = 0$, hence $\gamma_L\circ\phi_E = 0$. So, $\Im\phi_E \subset \Ker \gamma_L = \Im \phi_P$,
the last equality by Theorem~\ref{addone}~$(iii)$. This shows that the right-hand side is contained
in the left-hand-side. On the other hand, it is clear that $P \in \CA_X$,
hence the left-hand-side is contained in the right-hand-side.

The second part follows immediately from~Theorem~\ref{addone}~$(i)$.
Indeed, the map
$$
\gamma_P \oplus \gamma_L :\HOH_\bullet(X) \to \HOH_\bullet(X,P) \oplus \HOH_\bullet(X,L)
$$
is an isomorphism. Hence $\gamma_P$ gives an isomorphism of $\Ker\gamma_L$ onto $\HOH_\bullet(X,P)$.
But $\Ker\gamma_L = \Im\phi_P$ by~Theorem~\ref{addone}~$(iii)$.
\end{proof}

From now on we {\em identify}\/ the Hochschild homology $\HOH_\bullet(\CA)$
of an admissible subcategory $\CA \subset \D^b(X)$ with the subspace $\Im\phi_P$
of $\HOH_\bullet(X)$. We want to summarize the properties of these subspaces.

\begin{corollary}\label{hohadd}
For any semiorthogonal decomposition $\D^b(X) = \lan \CA_1,\dots,\CA_n \ran$
there is a direct sum decomposition
$\HOH_\bullet(X) = \HOH_\bullet(\CA_1) \oplus \dots \oplus \HOH_\bullet(\CA_n)$.
Moreover, for the kernel $P_i$ of the projection onto $\CA_i$ we have
\begin{enumerate}
\item $\gamma_{P_i}(\HOH_\bullet(\CA_j)) = 0$ for $i\ne j$;
\item $\gamma_{P_i}:\HOH_\bullet(\CA_i) \to \HOH_\bullet(X,P_i)$ is an isomorphism;
\item $\phi_{P_i}$ is a projector onto $\HOH_\bullet(\CA_i)$ along $\oplus_{j\ne i}\HOH_\bullet(\CA_j)$.
\end{enumerate}
\end{corollary}

Now we can construct explicit isomorphisms for Hochschild homology.

\begin{theorem}\label{hhhcor}
If a right splitting functor $\Phi_\CE:\D^b(X) \to \D^b(Y)$
gives an equivalence of subcategories $\CA \subset \D^b(X)$ and $\CB \subset \D^b(Y)$
then the map $\phi_\CE:\HOH_\bullet(X) \cong \HOH_\bullet(Y)$ induces an isomorphism
of Hochschild homologies $\HOH_\bullet(\CA) \subset \HOH_\bullet(X)$ and
$\HOH_\bullet(\CB) \subset \HOH_\bullet(Y)$.
\end{theorem}
\begin{proof}
Recall that by Lemma~\ref{adjs} there is a kernel $\CE^! \in \D^b(X\times Y)$ which
gives a right adjoint functor to~$\Phi_\CE$. Since $\Phi_\CE$ is a splitting functor, the compositions
$\Phi_{\CE^!}\circ\Phi_\CE$ and $\Phi_\CE\circ\Phi_{\CE^!}$ are the projections functors
onto subcategories $\CA \subset \D^b(X)$ and $\CB \subset \D^b(Y)$ respectively,
so by Theorem~\ref{sod-kernels} we have
$$
P_\CA := \Phi_{\CE^!}\circ\Phi_\CE
\qquad\text{and}\qquad
P_\CB := \Phi_\CE\circ\Phi_{\CE^!}
$$
are the kernels giving these projection functors.
By Lemma~\ref{phik} the functors $\Phi_\CE$ and $\Phi_{\CE^!}$ induce maps
$\phi_\CE:\HOH_\bullet(X) \to \HOH_\bullet(Y)$ and $\phi_{\CE^!}:\HOH_\bullet(Y) \to \HOH_\bullet(X)$.
Moreover, $\phi_{\CE^!}\circ\phi_\CE = \phi_{\CE^!\circ\CE} = \phi_{P_\CA}$,
$\phi_{\CE}\circ\phi_{\CE^!} = \phi_{\CE\circ\CE^!} = \phi_{P_\CB}$,
which by Corollary~\ref{hohadd} are the projectors onto $\HOH_\bullet(\CA) \subset \HOH_\bullet(X)$
and $\HOH_\bullet(\CB) \subset \HOH_\bullet(Y)$ respectively. Hence
$\phi_\CE$ and $\phi_{\CE^!}$ induce isomorphisms of $\HOH_\bullet(\CA)$ and $\HOH_\bullet(\CB)$.
\end{proof}

\subsection{Exact sequences for Hochschild cohomology}

We close this section by a discussion of the relation of the Hochschild cohomology of the components
of a semiorthogonal decomposition of a triangulated category to the its Hochschild cohomology.


\begin{theorem}\label{chh}
Let $\CA \subset \D^b(X)$ be an admissible subcategory and $\CA = \langle \CA_1,\CA_2 \rangle$
a semiorthogonal decomposition. Let $P,P_1,P_2 \in \D^b(X\times X)$ be the kernels of the
projection functors onto $\CA,\CA_1,\CA_2$. Then there is an exact sequence
\begin{equation}\label{chh1}
\dots \to \HOH^t(\CA) \to \HOH^t(\CA_1) \oplus \HOH^t(\CA_2) \to \Hom^{t+1}(P_1,P_2) \to \HOH^{t+1}(\CA) \to \dots.
\end{equation}
Moreover, let $P'_2$ be the kernel of the left projection onto $\CA_2$ in $\CA$.
Then there is an exact sequence
\begin{equation}\label{chh2}
\dots \to \HOH^t(\CA) \to \HOH^t(\CA_1) \to \Hom^{t+1}(P'_2,P_2) \to \HOH^{t+1}(\CA) \to \dots.
\end{equation}
\end{theorem}
\begin{proof}
We extend the semiorthogonal decomposition of $\CA$ to a decomposition of $\D^b(X)$ as
$$
\D^b(X) = \langle \CA^\perp,\CA_1,\CA_2 \rangle.
$$
Then we have a distinguished triangle
$$
P_2 \to P \to P_1 \to P_2[1].
$$
Since $\Hom^\bullet(P_1,P_2) = 0$ (see the proof of Corollary~\ref{so1}), we obtain two exact sequences
$$
\dots \to \Hom^t(P,P) \to \Hom^t(P_1,P_1) \to \Hom^{t+1}(P,P_2) \to \Hom^{t+1}(P,P) \to \dots
$$
and
$$
\dots \to \Hom^t(P,P) \to \Hom^t(P_1,P_1) \oplus \Hom^t(P_2,P_2) \to \Hom^{t+1}(P_1,P_2) \to \Hom^{t+1}(P,P) \to \dots.
$$
Recalling the definition of the Hochschild cohomology of an admissible subcategory,
we see that the second sequence gives~\eqref{chh1}.

Further, by definition of $P'_2$ we have a distinguished triangle
$$
P_3 \to P \to P'_2
$$
where $P_3$ is the projection onto ${}^\perp\CA_2$. In particular, $P_3 \in ({}^\perp\CA_2)_X$,
hence $\Hom^\bullet(P_3,P_2) = 0$. Therefore
$\Hom^\bullet(P,P_2) \cong \Hom^\bullet(P'_2,P_2)$,
so the first sequence gives~\eqref{chh2}.
\end{proof}

\begin{remark}
Let $\alpha_i:\CA_i \to \CA$ be the embedding functors.
Recall the functor
$$
\phi = \alpha_2^*\circ\alpha_1 : \CA_1 \to \CA_2,
$$
known as {\sf the gluing functor}\/ of the semiorthogonal decomposition.
One can show that
$$
\Hom^{\bullet+1}(P_1,P_2) \cong \Hom^\bullet(\phi,\phi),
$$
so the third term of~\eqref{chh1} can be thought of as endomorphisms of the gluing functor.
\end{remark}

\begin{remark}
It would be interesting to give an interpretation of the term $\Hom^\bullet(P_2',P_2)$
of the sequence~\eqref{chh2}. One can show that if $\CA_2 \cong \D^b(Y)$ and the equivalence
is given by a kernel functor $\Phi_K:\D^b(Y) \to \D^b(X)$ then
$\Hom^\bullet(P_2',P_2) \cong \Hom^\bullet(\Delta_*\CO_Y,K^!\circ S_\CA^{-1}\circ K\circ S_Y)$.
In particular, if $Y = \Spec\kk$, so that $K$ considered as an object of $\D^b(X\times\Spec\kk) = \D^b(X)$
is an exceptional object, then $\Hom^\bullet(P_2',P_2) \cong \Hom^t(K,S_{\CA}^{-1}(K))$.
\end{remark}

\section{Examples of Hochschild homology and cohomology of admissible subcategories}\label{s-comp}

Let $X$ be a smooth projective variety of dimension $n$.
We start with the following simple observation.

\begin{proposition}\label{comp-hh}
Let $\CA \subset \D^b(X)$ be an admissible subcategory.
Let $P$ be the kernel of the left projection to $\CA$.
Then
$$
\HOH^\bullet(\CA) = \HH^\bullet(X,\Delta^! P),
\qquad
\HOH_\bullet(\CA) = \HH^\bullet(X,\Delta^* P).
$$
\end{proposition}
\begin{proof}
Indeed, let $R$ be the kernel of the right projection onto ${}^\perp\CA$.
Then we have a distinguished triangle $R \to \Delta_*\CO_X \to P$.
On the other hand, $\Hom^\bullet(R,P) = \Hom^\bullet(R,P\circ S_X) = 0$ by Proposition~\ref{ki} and Corollary~\ref{so}.
Hence we have
$$
\begin{array}{l}
\Hom^\bullet(P,P) \cong \Hom^\bullet(\Delta_*\CO_X,P) \cong \Hom^\bullet(\CO_X,\Delta^!P),\\
\Hom^\bullet(P,P\circ S_X) \cong \Hom^\bullet(\Delta_*\CO_X,P\circ S_X) \cong \Hom^\bullet(\CO_X,\Delta^!(P\circ S_X)).
\end{array}
$$
and it remains to note that $\Delta^!(P\circ S_X) \cong \Delta^!(P\otimes p_2^*\omega_X[\dim X]) \cong \Delta^*P$.
\end{proof}

In the case $\CA = \D^b(X)$ we obtain
$$
\HOH^\bullet(X) = \HH^\bullet(X,\Delta^!\Delta_*\CO_X),
\qquad
\HOH_\bullet(X) = \HH^\bullet(X,\Delta^*\Delta_*\CO_X).
$$
To compute the RHS explicitly one uses the Atiyah classes.
Recall that {\sf the universal Atiyah class of $X$}\/ is the morphism
$$
\At:\Delta_*\CO_X \to \Delta_*\Omega_X[1]
$$
corresponding to the extension
\begin{equation}\label{at-seq}
0 \to I/I^2 \to \CO_{X\times X}/I^2 \to \Delta_*\CO_X \to 0,
\end{equation}
where $I \subset \CO_{X\times X}$ is the sheaf of the ideals of the diagonal
(note that $I/I^2$ is the conormal bundle for the diagonal, so it is canonically
isomorphic to $\Delta_*\Omega_X$).

\begin{remark}\label{at-e}
Considering $\Delta_*\CO_X$ and $\Delta_*\Omega_X[1]$ as kernels, we see that
the universal Atiyah class induces a morphism of kernel functors
$\Phi_{\Delta_*\CO_X} \to \Phi_{\Delta_*\Omega_X[1]}$. Evaluating it on an object $E \in \D^b(X)$
we obtain a map $\At_E:E \to E\otimes\Omega_X[1]$ which is known as {\sf the Atiyah class of $E$}.
Note that $\At_E$ is functorial with respect to $E$.
Note also that for a trivial vector bundle $E = \CO_X$ the Atiyah class $\At_{\CO_X}$ is zero.
Indeed to check this one should take the pushforward to $X$ of the sequence~\eqref{at-seq}
and to check that it splits (a splitting is given by the morphism $\CO_X \to p_{1*}(\CO_{X\times X}/I^2)$
corresponding by adjunction to the map $p_1^*\CO_X = \CO_{X\times X} \to \CO_{X\times X}/I^2$).
\end{remark}

Iterating the Atiyah class we obtain the maps
$$
\At^p:\Delta_*\CO_X \to \Delta_*\Omega^p_X[p],
\qquad
\At^p:\Delta_*\Omega^{n-p}_X[n-p] \to \Delta_*\omega^n_X[n],
$$
for all $p$. By adjunction they give maps $\Delta^*\Delta_*\CO_X \to \Omega^p_X[p]$
and $\Lambda^pT_X[-p] \to \Delta^!\Delta_*\CO_X$.

\begin{theorem}[\cite{Ma,Ma2}]\label{hkr}
The maps
$$
\xymatrix@1{\Delta^*\Delta_*\CO_X \ar[rr]^-{\oplus \At^p} && \bigoplus\limits_{p=0}^n \Omega^p_X[p]},
\qquad
\xymatrix@1{\bigoplus\limits_{p=0}^n \Lambda^pT_X[-p] \ar[rr]^-{\oplus \At^p} && \Delta^!\Delta_*\CO_X}
$$
are isomorphisms. In particular
$$
\HOH^\bullet(X) \cong \bigoplus_{p=0}^n H^{t-p}(X,\Lambda^pT_X),
\qquad
\HOH_\bullet(X) \cong \bigoplus_{p=0}^n H^{p+t}(X,\Omega^p_X).
$$
\end{theorem}
The isomorphisms of the Theorem are known as the {\sf Hochschild-Kostant-Rosenberg}\/
({\sf HKR}\/ for short) {\sf isomorphisms}.

\begin{corollary}\label{hohk}
Let $E \in \D^b(X)$ be an exceptional object. Then $\HOH_\bullet(\lan E\ran) = \HOH^\bullet(\lan E\ran) = \kk$.
\end{corollary}

The goal of this section is to find a generalization
of this result for some admissible subcategories.

\subsection{The orthogonal to the structure sheaf}

The following result applies to any Fano variety.

\begin{theorem}\label{chh-operp}
Assume that $X$ is a smooth projective variety such that
$\CO_X$ is an exceptional bundle.
Then we have
$$
\HOH_\bullet(\CO_X^\perp) = \HOH_\bullet(X) / \kk,
\qquad
\HOH^t(\CO_X^\perp) = \oplus_{p=0}^{n-1} H^{t-p}(X,\Lambda^pT_X).
$$
\end{theorem}
\begin{proof}
The first is evident by the Additivity Theorem and Proposition~\ref{hohk}. So we have to check the second.
By Proposition~\ref{comp-hh} we have to compute $\Delta^!P$, where $P$
is the kernel of the left projection onto~$\CO_X^\perp$.
Since the kernel of the right projection onto $\langle \CO_X \rangle$
is given by the sheaf $\CO_X\boxtimes\CO_X \in \D^b(X\times X)$, we conclude
that we have a distinguished triangle
$$
\CO_X\boxtimes\CO_X \to \Delta_*\CO_X \to P.
$$
Applying the functor $\Delta^!$ and taking into account Theorem~\ref{hkr} we obtain the triangle
$$
\omega_X^{-1}[-n] \to \mathop{\oplus}_{p=0}^n \Lambda^pT_X[-p] \to \Delta^!P.
$$
Since the HKR isomorphism is given by the universal Atiyah class,
and the Atiyah class of $\CO_X$ is trivial, it follows that
the first map in the above triangle is an isomorphism onto the $n$-th summand.
Thus $\Delta^!P \cong \mathop{\oplus}_{p=0}^{n-1} \Lambda^pT_X[-p]$ and we are done.
\end{proof}

\subsection{The orthogonal to an exceptional pair}

Now assume that there is an exceptional vector bundle $E$ on $X$ right orthogonal to $\CO_X$,
that is $E \in \CO_X^\perp$. Then the same argument can be applied once again.
Before stating the result we introduce some notation.

Let $V^\bullet = \HH^\bullet(X,E^\vee)^\vee$ and let $E^\perp$ be the left mutation of $E^\vee$ through $\CO_X$ shifted by $-1$.
By Theorem~\ref{muts} we have the following distinguished triangles
$$
E^\perp \to V^\vee\otimes\CO_X \to E^\vee,
\qquad
E \to V\otimes\CO_X \to E^{\perp\vee}.
$$
Recall that by functoriality of the Atiyah class we have a commutative diagram
$$
\xymatrix{
&
E \ar[r] \ar[d]^{\At_{E}} \ar@{..>}[dl]_{\alpha_E} &
V\otimes\CO_X \ar[d]^{\At_{V\otimes\CO_X}} \\
E^{\perp\vee}\otimes\Omega_X \ar[r] &
E\otimes\Omega_X[1] \ar[r] &
V\otimes\Omega_X[1]
}
$$
But $\At_{V\otimes\CO_X} = 0$ (see Remark~\ref{at-e}), so it follows that the Atiyah class $\At_{E}$ factors
as a composition $E \to E^{\perp\vee}\otimes\Omega_X \to E\otimes\Omega_X[1]$ for some
$\alpha_E \in \Hom(E,E^{\perp\vee}\otimes\Omega_X) \cong \Hom(E^\perp\otimes E,\Omega_X)$
(see the proof of Theorem~\ref{chhort2} for a more invariant construction of $\alpha_E$).

\begin{remark}
If $E^\vee$ is a globally generated vector rank $r$ bundle without higher cohomology
then there is a map $\phi:X \to \Gr(r,V)$ such that $E$ is the pullback of the tautological bundle.
Then $E^\perp\otimes E \cong \phi^*\Omega_{\Gr(r,V)}$ and the map $\alpha_E$ coincides with
$d\phi:\phi^*\Omega_{\Gr(r,V)} \to \Omega_X$.
\end{remark}

Combining $\alpha_E$ with the powers of the Atiyah class of $E$ we obtain a collection of maps
$$
\xymatrix@1{
E^\perp\otimes E \ar[r]^-{\At^{i-1}_E} &
E^\perp\otimes E\otimes\Omega^{i-1}_X[i-1] \ar[r]^-{\alpha_E} &
\Omega_X\otimes\Omega^{i-1}_X[i-1] \ar[r] &
\Omega^i_X[i-1]
}
$$
which we denote by $\alpha^i_E$. Further, let $\CN^\vee$ denote the cone of the map
$\alpha_E:E^\perp\otimes E \to \Omega_X$ shifted by $-1$, so that we have a distinguished triangle
$$
\CN^\vee \to E^\perp\otimes E \to \Omega_X.
$$
The composition of $\alpha^i_X$ with the map $\CN^\vee \to E^\perp\otimes E$ is a map
$\CN^\vee \to \Omega^i_X[i-1]$ which we denote by $\nu^i_E$.

\begin{remark}
If $E$ is the pullback of the tautological bundle via a map $\phi:X \to \Gr(r,V)$
and $\phi$ is a closed embedding then $\CN^\vee$ is the conormal bundle.
\end{remark}


\begin{theorem}\label{chhort2}
The Hochschild homology of the category $\langle E,\CO_X \rangle^\perp$ is given by the formula
$$
\HOH_t(\langle E,\CO_X \rangle^\perp) = \begin{cases} \HOH_t(X), & \text{if $t \ne 0$}\\ \HOH_0(X)/\kk^2, & \text{if $t = 0$}\end{cases}
$$
For the Hochschild cohomology we have the following results:

\noindent
$(i)$ there is an exact sequence
$$
\dots \!\!\to
\mathop{\oplus}\limits_{p=0}^{n-1} H^{t-p}(X,\Lambda^pT_X) \!\to
\HOH^t(\langle E,\CO_X \rangle^\perp) \!\to
H^{t-n+2}(X,E^\perp\otimes E\otimes\omega_X^{-1}) \stackrel{\alpha}\to
\mathop{\oplus}\limits_{p=0}^{n-1} H^{t+1-p}(X,\Lambda^pT_X) \to\! \dots,
$$
where
$\alpha = \sum_{p=0}^{n-1}\alpha^{n-1-p}_E:
E^\perp\otimes E\otimes\omega_X^{-1} \to
\Omega^{n-p}_X\otimes\omega_X^{-1} =
\Lambda^pT_X$;

\noindent
$(ii)$ there is an exact sequence
$$
\dots \!\to
\mathop{\oplus}\limits_{p=0}^{n-2} H^{t-p}(X,\Lambda^pT_X) \to
\HOH^t(\langle E,\CO_X \rangle^\perp) \to
H^{t-n+2}(X,\CN^\vee_{X/\Gr}\otimes\omega_X^{-1}) \stackrel{\nu}\to
\mathop{\oplus}\limits_{p=0}^{n-2} H^{t+1-p}(X,\Lambda^pT_X) \to\!
\dots,
$$
where
$\nu = \sum_{p=0}^{n-2}\nu^{n-1-p}_E:
\CN^\vee\otimes\omega_X^{-1} \to
\Omega^{n-p}_X\otimes\omega_X^{-1} =
\Lambda^pT_X$;

\noindent
$(iii)$ if $E$ is a line bundle then $\nu = 0$ and
$$
\HOH^t(\langle E,\CO_X \rangle^\perp) \cong
\mathop{\oplus}\limits_{p=0}^{n-2} H^{t-p}(X,\Lambda^pT_X) \oplus H^{t-n+2}(X,\CN^\vee_{X/\Gr}\otimes\omega_X^{-1}).
$$
\end{theorem}
\begin{proof}
The first is evident by additivity theorem. So we have to check the second.
Let $P_1$ be the kernel of the left projection onto $\CO_X^\perp$
and $P_2$ be the kernel of the left projection onto $\langle E,\CO_X \rangle^\perp$.
We have a distinguished triangle $\CO_X\boxtimes\CO_X \to \Delta_*\CO_X \to P_1$.
On the other hand, the object $E\boxtimes E^\vee$ is the kernel of the right projection onto $\langle E\rangle$.
Taking the convolution of the above triangle with the canonical morphism $E\boxtimes E^\vee \to \Delta_*\CO_X$
we obtain a commutative diagram
$$
\xymatrix{
E\boxtimes (V^\vee\otimes\CO_X) \ar[r] \ar[d] &
E\boxtimes E^\vee \ar[r] \ar[d] &
E\boxtimes E^\perp[1] \ar[d] \\
\CO_X\boxtimes\CO_X \ar[r] &
\Delta_*\CO_X \ar[r] &
P_1
}
$$
(indeed, it is easy to see that $(E\boxtimes E^\vee)\circ(\CO_X\boxtimes\CO_X) \cong E\boxtimes (V^\vee\otimes\CO_X)$
by Lemma~\ref{conv-decomp}, and the induced map $E\boxtimes (V^\vee\otimes\CO_X) \to E\boxtimes E^\vee$
is the canonical map $V^\vee\otimes\CO_X \to E^\vee$ tensored with $E$, hence the third vertex of the upper line
is $E\boxtimes E^\perp$). So, it follows that $E\boxtimes E^\perp = (E\boxtimes E^\vee)\circ P_1$,
hence the corresponding kernel functor is the composition of the left projection onto $\CO_X^\perp$
and the right projection onto $\langle E\rangle$. Hence the cone of the map $E\boxtimes E^\perp[1] \to P_1$
is the kernel of the left projection onto $\langle E,\CO_X \rangle^\perp$, which is $P_2$.
Thus the right vertical map extends to a distinguished triangle
$$
E\boxtimes E^\perp[1] \to P_1 \to P_2.
$$
By Proposition~\ref{comp-hh} we have to compute $\Delta^!P_2$.
Applying the functor $\Delta^!$ to the above diagram we obtain
$$
\xymatrix{
V^\vee\otimes E\otimes\omega_X^{-1}[-n] \ar[r] \ar[d] &
E\otimes E^\vee\otimes\omega_X^{-1}[-n] \ar[r] \ar[d] &
E\otimes E^\perp\otimes\omega_X^{-1}[1-n] \ar[d] \\
\omega_X^{-1}[-n] \ar[r] &
\mathop{\oplus}_{p=0}^{n} \Lambda^pT_X[-p] \ar[r] &
\mathop{\oplus}_{p=0}^{n-1} \Lambda^pT_X[-p]
}
$$
Here the middle vertical arrow is the sum of the Atiyah classes of $E$.
By the commutativity of the diagram the right arrow gives their
factorizations $\alpha^i_E$. This proves part~$(i)$ of the Theorem.
Part~$(ii)$ follows immediately by definition of $\CN^\vee$ and $\nu$.
Finally, for $(iii)$ we note that if $E$ is a line bundle then $\At^{i-1}_E \in H^{i-1}(X,\Omega^{i-1}_X)$,
hence we have a commutative diagram
$$
\xymatrix{
\CN^\vee \ar[r] &
E^\perp\otimes E \ar[r]^{\alpha_E} \ar[d]_{\At_E^{i-1}} &
\Omega_X \ar[d]_{\At_E^{i-1}} \\
&
E^\perp\otimes E\otimes\Omega_X^{i-1}[i-1] \ar[r]^-{\alpha_E} &
\Omega_X\otimes\Omega_X^{i-1}[i-1] \ar[r] &
\Omega_X^{i}[i-1]
}
$$
The map $\nu^i_E$ by definition is the composition of the maps in the diagram.
But the composition in the upper line is zero by the definition of $\CN^\vee$.
Hence $\nu^i_E = 0$ and we are done.
\end{proof}

\subsection{Fano threefolds of index $2$}

Now we are going to apply these results to Fano threefolds.
In the following Theorem we remind the classification of Fano threefolds
of index 2 (see e.g.~\cite{IP} for details) and compute the Hochschild cohomology
of the nontrivial component of their derived categories.

\begin{theorem}
Let $X = X_d$ be a Fano threefold of index $2$ with $\Pic X = \ZZ$ and the degree $d$,
so that $(\CO_X(-1),\CO_X)$ is an exceptional pair.
Let $\CA_X = \langle \CO_X(-1),\CO_X \rangle^\perp$ be the orthogonal subcategory.
\begin{itemize}
\item if $d = 5$ and $X \subset \Gr(2,W)$, $\dim W = 5$, is the zero locus
of a regular section of a vector bundle $A^\vee \otimes \CO_{\Gr(2,W)}(1)$, $\dim A = 3$,
corresponding to an embedding $A \to \Lambda^2W^\vee$, then
$$
\HOH^p(\CA_X) =
\begin{cases}
\kk, & \text{if $p = 0$},\\
\fsl(A) \cong \kk^8, & \text{if $p = 1$},\\
0, & \text{if $p \ne 0,1$}.
\end{cases}
$$
\item if $d = 4$ and $X \subset \PP(V)$, $\dim V = 6$, is the zero locus
of a regular section of a vector bundle $A^\vee \otimes \CO_{\PP(V)}(2)$, $\dim A = 2$,
corresponding to an embedding $A \to S^2V^\vee$, then
$$
\HOH^p(\CA_X) =
\begin{cases}
\kk, & \text{if $p = 0$},\\
A \cong \kk^2, & \text{if $p = 1$},\\
S^2A \cong \kk^3, & \text{if $p = 2$},\\
0, & \text{if $p \ne 0,1,2$}.
\end{cases}
$$
\item if $d = 3$ and $X \subset \PP(V)$, $\dim V = 5$ is a cubic hypersurface then
$$
\HOH^p(\CA_X) =
\begin{cases}
\kk, & \text{if $p = 0$},\\
S^3V^\vee/\fgl(V) \cong \kk^{10}, & \text{if $p = 2$},\\
0, & \text{if $p \ne 0,2$}.
\end{cases}
$$
\item if $d = 2$ and $X \stackrel{\phi}\to \PP(V)$, $\dim V = 4$,
is the double covering ramified in a quartic surface, then
$$
\HOH^p(\CA_X) =
\begin{cases}
\kk, & \text{if $p = 0$},\\
S^4V^\vee/\fgl(V) \oplus \kk \cong \kk^{20}, & \text{if $p = 2$},\\
\kk, & \text{if $p = 4$},\\
0, & \text{if $p \ne 0,2,4$}.
\end{cases}
$$
\item if $d = 1$ and $X \subset \PP(1,1,1,2,3)$ is a degree $6$ hypersurface
in the weighted projective space, then
$$
\HOH^p(\CA_X) =
\begin{cases}
\kk, & \text{if $p = 0$},\\
\kk^{34} \oplus \kk \cong \kk^{35}, & \text{if $p = 2$},\\
S^2\kk^3 \oplus \kk \cong \kk^7, & \text{if $p = 4$},\\
0, & \text{if $p \ne 0,2,4$}.
\end{cases}
$$
\end{itemize}
\end{theorem}
\begin{proof}
By Theorem~\ref{chhort2} we should compute $H^\bullet(X,\CO_X)$, $H^\bullet(X,T_X)$ and $H^\bullet(X,\CN^\vee(2))$.
We will use a case-by-case analysis.

{\bf The case $d = 5$.}\/ It is easy to see that $\CN^\vee \cong W/\CU(-2)$, where $\CU$ is the restriction to $X$
of the tautological bundle on $\Gr(2,W)$. So, $\CN^\vee(2) \cong W/\CU$ and it is easy to compute
$H^\bullet(X,W/\CU) \cong W$. On the other hand, it is well known that $H^\bullet(X,\CO_X) = \kk$
and $H^\bullet(X,T_X) = \fso(A^\vee,q)$, where $q \in S^2A$ is contained in the kernel of the natural map
$S^2A \to S^2(\Lambda^2 W^\vee) \to \Lambda^4 W^\vee \cong W$
induced by the embedding $A \to \Lambda^2W^\vee$.
Applying exact sequence of Theorem~\ref{chhort2}(ii) we obtain $\HOH^0(\CA_X) = \kk$, $\HOH^{\ne 0,1}(\CA_X) = 0$
and an exact triple
$$
0 \to \fso(A^\vee,q) \to \HOH^1(\CA) \to W \to 0.
$$
Finally, it is easy to see that this sequence can be identified with the sequence
$$
0 \to \Lambda^2 A \to (A\otimes A)/\lan q \ran \to W \to 0,
$$
and that an isomorphism $q: A^\vee \to A$ identifies its middle term with $\fsl(A)$.

\begin{remark}
An alternative computation of $\HOH^\bullet(\CA_X)$ can be given by an explicit description of $\CA_X$.
Indeed, it is well known (see~\cite{O-V5}) that $(\CU(-1),\CU^\perp(-1),\CO_X(-1),\CO_X)$ is a full exceptional collection
on $X = X_5$, hence $\CA_X = \langle \CU(-1),\CU^\perp(-1) \rangle$ is equivalent to the derived category
of a quiver with 2 vertices and the space of arrows $A$.
\end{remark}

{\bf The case $d = 4$.}\/ It is evident that $\CN^\vee \cong A\otimes\CO_X(-2)$.
So, $\CN^\vee(2) \cong A\otimes\CO_X$ and $H^\bullet(X,A\otimes\CO_X) \cong A$.
On the other hand, it is well known that $H^\bullet(X,\CO_X) = \kk$
and $H^\bullet(X,T_X) = (A^\vee\otimes (S^2V^\vee/A))/\fsl(V) [-1]$.
Moreover, one can show that the later space is just $S^2A[-1]$.
Indeed, we have a canonical map
$$
\frac{A^\vee \otimes S^2V}{\fgl(A) \oplus \fsl(V)} \to \frac{S^6A^\vee}{\fgl(A)} \to S^2A,
$$
where the first arrow is the differential of the determinant (with respect to $V$)
and the second map is the differential of the covariant of $\fsl(A)$.
It is a straightforward computation that the map is surjective,
hence it is an isomorphism.
Now we apply part $(iii)$ of Theorem~\ref{chhort2} and obtain the result.

\begin{remark}
An alternative computation of $\HOH^\bullet(\CA_X)$ can be given by an explicit description of $\CA_X$.
Indeed, it is known (see~\cite{BO1,K-Q}) that $\CA_X\cong \D^b(C)$, where $\psi:C \to \PP(A)$ is the double covering
ramified in the points of $\PP(A)$ corresponding to degenerate quadrics in the pencil $A \subset S^2V^\vee$
(thus $C$ is a curve of genus~$2$). We have
$\psi_*\CO_C \cong \CO_{\PP(A)} \oplus \CO_{\PP(A)}(-3)$,
$\psi_*T_C \cong \CO_{\PP(A)}(-1) \oplus \CO_{\PP(A)}(-4)$,
and Theorem~\ref{hkr} applies.
\end{remark}

{\bf The case $d = 3$.}\/ It is evident that $\CN^\vee \cong \CO_X(-3)$.
So, $\CN^\vee(2) \cong \CO_X(-1)$ and $H^\bullet(X,\CO_X(-1)) = 0$.
On the other hand, it is well known that $H^\bullet(X,\CO_X) = \kk$
and $H^\bullet(X,T_X) = (S^3V^\vee)/\fgl(V) [-1]$.
Now we apply part $(iii)$ of Theorem~\ref{chhort2} and obtain the result.

{\bf The case $d = 5$.}\/ We have an exact sequence
$$
0 \to \phi^*\Omega_{\PP(V)} \to \Omega_X \to i_*\CN_{D/X} \to 0,
$$
where $i:D \to X$ is the ramification divisor of the double covering and $\CN_{D/X}$ is the normal bundle.
Thus $\CN^\vee = i_*\CN_{D/X}[-1]$. Further, the map $\phi\circ i:D \to \PP(V)$
is an isomorphism of $D$ with the quartic hypersurface in $\PP(V)$
(which by an abuse of notation we also denote by $D$) and $\CN_{D/X}$
gets identified with $\CO_D(-2)$. Thus
$H^\bullet(X,\CN^\vee(2)) = H^\bullet(D,\CO_D[-1]) = \kk[-1] \oplus \kk[-3]$.
On the other hand, it is well known that $H^\bullet(X,\CO_X) = \kk$
and $H^\bullet(X,T_X) = (S^4V^\vee)/\fgl(V) [-1]$.
Now we apply part $(iii)$ of Theorem~\ref{chhort2} and obtain the result.

{\bf The case $d = 5$.}\/ The sheaf $\Omega_X$ is the middle cohomology of a complex
$$
0 \to \CO_X(-6) \to \CO_X(-3) \oplus \CO_X(-2) \oplus \CO_X(-1)^{\oplus 3} \to \CO_X \to 0,
$$
while for $E^\perp\otimes E$ we have a resolution
$$
0 \to E^\perp\otimes E \to \CO_X(-1)^{\oplus 3} \to \CO_X \to 0,
$$
and the map $\alpha_E$ is given by the natural embedding of complexes.
Therefore we have a distinguished triangle
$$
\CN^\vee \to \CO_X(-6) \to \CO_X(-3) \oplus \CO_X(-2).
$$
Twisting it by $\CO_X(2)$ and computing the cohomology we obtain
$H^\bullet(X,\CN^\vee(2)) = \kk[-1] \oplus (S^2\kk^3\oplus\kk)[-3]$.
On the other hand, it is well known that $H^\bullet(X,\CO_X) = \kk$
and $H^\bullet(X,T_X) = \kk^{34} [-1]$.
Now we apply part~$(iii)$ of Theorem~\ref{chhort2} and obtain the result.
\end{proof}

\subsection{Conic bundles}

Another case we consider is the case of a conic bundle.
Recall that a conic bundle is a flat projective morphism
$f:X \to Y$ each fiber of which is isomorphic (as a scheme) to a conic
(possibly degenerate) in~$\PP^2$. Each conic bundle can be embedded into
a projectivization $p:\PP_Y(E) \to E$ of some rank $3$ vector bundle $E$ on $Y$
(e.g. $E = (f_*\omega_{X/Y}^{-1})^\vee$) and can be represented as a zero locus
of a global section of a line bundle $p^*L\otimes\CO_{\PP_Y(E)/Y}(2)$ on $X$
for an appropriate line bundle $L$ on $Y$. This global section can be thought as a section
of the vector bundle $L\otimes S^2E^*$ on $Y$, and so gives a morphism $E \to L\otimes E^*$.
The zero locus $D$ of the determinant $\det E \to L^3\otimes\det E^*$ of this map
is the discriminant locus of $f:X \to Y$. For any point $y \in Y\setminus D$ the fiber
$X_y$ is a smooth conic, while for $y \in D$ the fiber $X_y$ is degenerate conic
(either a union of two lines, or a double line).
Moreover, it is well known that if both $X$ and $Y$ are smooth then the nonreduced
fivers of $f$ correspond to singular points of $D$. Below we assume that $D$ is smooth
and so all singular fibers are the unions of two distinct lines. In this case there is
an embedding $j:D \to X$ each point $y \in D$ goes to the singular point of the fiber $X_y$
(the intersection point of the corresponding two lines). Moreover, in this case the Stein factorization
of $f^{-1}(D) \to D$ gives a nonramified double covering $\tilde{D} \to D$. Let $M \in \Pic^0(D)$
be the corresponding point of order $2$.

For any conic bundle $f:X \to Y$ the pullback functor $f^*:\D^b(Y) \to \D^b(X)$ is fully faithful
and gives rise to a semiorthogonal decomposition $\D^b(X) = \langle \CA_X, f^*(\D^b(Y)) \rangle$.
Actually, $\CA_X \cong \D^b(Y,\CB_0)$, the derived category of sheaves of modules over the sheaf
of even parts of the Clifford algebras of conics (see~\cite{K-Q} for further details).
In the following Theorem we compute the Hochschild homology and cohomology of the category $\CA_X$.

\begin{theorem}
Let $f:X \to Y$ be a conic bundle. Let $i:D \to Y$ be the discriminant locus of~$f$.
Assume that $X$, $Y$, and $D$ are smooth, $\dim X = n$.
Then
$$
\HOH_t(\CA_X) \cong
\HOH_t(X)/\HOH_\bullet(Y) \cong
\HOH_t(Y) \oplus \left( \bigoplus_{p=0}^{n-2}\, H^{p+t}(D,\Omega^p_D\otimes M) \right).
$$
Further, the Hochschild cohomology of $\CA_X$ is isomorphic to the cohomology of polyvector fields
on $Y$ tangent to $D$. More precisely,
$$
\HOH^t(\CA_X) = \bigoplus_{p=0}^n H^{t-p}(Y,\Ker(\Lambda^pT_Y \to i_*(\Lambda^{p-1}T_D\otimes\CN_{D/Y}))).
$$
\end{theorem}
\begin{proof}
Consider the natural map $T_X \to f^*T_Y$ (the differential of $f$).
Note that its kernel $T_{X/Y}$ is a reflexive sheaf of rank $1$, hence an invertible sheaf.
On the other hand, its cokernel is supported at $j(D)$. An easy computation shows that
actually we have an exact sequence
\begin{equation}\label{tan}
0 \to \omega^{-1}_{X/Y} \to T_X \to f^*T_Y \to j_*\CN_{D/Y} \to 0.
\end{equation}
Dualizing it and taking into account an isomorphism $j^*\omega_{Y/X} \cong M$ we obtain an exact sequence
\begin{equation}\label{cotan}
0 \to f^*\Omega_Y \to \Omega_X \to \omega_{X/Y} \to j_*M \to 0.
\end{equation}
Taking the $p$-th exterior powers we obtain exact sequences
$$
0 \to f^*(\Lambda^{p-1}T_Y)\otimes\omega_{X/Y}^{-1} \to \Lambda^{p}T_X \to f^*(\Lambda^{p}T_Y) \to j_*(\Lambda^{p-1}T_D\otimes\CN_{D/Y}) \to 0,
$$
$$
0 \to f^*\Omega^p_Y \to \Omega^p_X \to \omega_{X/Y}\otimes f^*\Omega^{p-1}_Y \to j_*(\Omega^{p-1}_D\otimes M) \to 0.
$$
Taking the pushforward to $Y$ we obtain an exact sequence
\begin{equation}\label{tan1}
0 \to \Lambda^{p-1}T_Y\otimes f_*(\omega_{X/Y}^{-1}) \to f_*(\Lambda^{p}T_X) \to \Lambda^{p}T_Y \to i_*(\Lambda^{p-1}T_D\otimes\CN_{D/Y}) \to 0
\end{equation}
as well as
\begin{equation}\label{cotan1}
R^0f_*\Omega^p_X = \Omega^p_Y,
\qquad
0 \to i_*(\Omega^{p-1}_D\otimes M) \to R^1f_*\Omega^p_X \to \Omega^{p-1}_Y \to 0.
\end{equation}
Moreover, it is easy to see that the sequence in~\eqref{cotan1} splits.
Now we can use these sequences to compute the Hochschild homology and cohomology of $\CA_X$.

For this we note that the right projection onto $\D^b(Y)$ embedded into $\D^b(X)$ via $f^*$ is given
by the kernel functor with the kernel being $\mu_*\CO_{X\times_Y X}$, where $\mu:X\times_Y X \to X\times X$
is the natural embedding (this follows from the base-change). Hence the kernel $P$ 
of the left projection functor onto $\CA_X$ is given by the following distinguished triangle
\begin{equation}\label{pax}
\mu_*\CO_{X\times_Y X} \to \Delta_*\CO_X \to P.
\end{equation}
Note also that $\mu_*\CO_{X\times_Y X} \cong (f\times f)^*\Delta_*\CO_Y$, hence
$$
\Delta^*\mu_*\CO_{X\times_Y X} \cong
\Delta^*(f\times f)^*\Delta_*\CO_Y \cong
f^*\Delta^*\Delta_*\CO_Y.
$$
So, applying to~\eqref{pax} the functor $\Delta^*$ and using Theorem~\ref{hkr} we obtain the following triangle
\begin{equation}\label{dpax}
\oplus_{p=0}^{n-1} f^*\Omega^p_Y[p] \to \oplus_{p=0}^n \Omega^p_X[p] \to \Delta^*P,
\end{equation}
with the left map being the direct sum of morphisms $f^*\Omega^p_Y \to \Omega^p_X$.
After the pushforward $f_*$ each of them goes to the isomorphism of~\eqref{cotan1}. Hence
$$
f_*\Delta^*P \cong
{\oplus}_{p=0}^n  R^1f_*\Omega^p_X[p-1]) \cong
{\oplus}_{p=1}^{n} (\Omega^{p-1}_Y \oplus i_*(\Omega^{p-1}_D\otimes M))[p-1].
$$
Taking the hypercohomology we obtain the required formula for $\HOH_\bullet(\CA_X)$.


On the other hand, note that $\Delta^!P \cong \Delta^*P\otimes\omega_X^{-1}[-n]$, hence
from~\eqref{dpax} we obtain
$$
\oplus_{p=0}^{n-1} \omega_X^{-1}\otimes f^*\Omega^p_Y[p-n] \to \oplus_{p=0}^n \omega_X^{-1}\otimes \Omega^p_X[p-n] \to \Delta^!P.
$$
Moreover, we have
$$
\omega_X^{-1}\otimes \Omega^p_X[p-n] \cong \Lambda^{n-p}T_X[p-n],
\qquad
\omega_X^{-1}\otimes f^*\Omega^p_Y[p-n] \cong \omega_{X/Y}^{-1}\otimes f^*\Lambda^{n-1-p}T_Y[p-n],
$$
hence the above triangle can be rewritten as
$$
\oplus_{p=0}^{n-1} \omega_{X/Y}^{-1}\otimes f^*\Lambda^{p-1}T_Y[-p] \to \Lambda^{p}T_X[-p] \to \Delta^!P.
$$
Taking the pushforward and using~\eqref{tan1} we obtain an isomorphism
$$
f_*\Delta^!P \cong \oplus_{p=0}^n \Ker(\Lambda^{p}T_Y \to  i_*(\Lambda^{p-1}T_D\otimes\CN_{D/Y})).
$$
Taking the hypercohomology we obtain the required formula for $\HOH^\bullet(\CA_X)$.
%
%
\end{proof}

\section{The Nonvanishing Conjecture}\label{sec-nvc}

We close the paper with the following Nonvanishing Conjecture.
In this Section we assume that $X$ is a smooth projective variety.

\begin{conjecture}
If $\CA \subset \D^b(X)$ is an admissible subcategory and $\HOH_\bullet(\CA) = 0$ then $\CA = 0$.
\end{conjecture}

This conjecture has several very pleasant corollaries.

\begin{corollary}
If $\CA_1,\dots,\CA_n \subset \D^b(X)$ is a semiorthogonal collection of admissible subcategories
such that $\oplus_{i=1}^n \HOH_\bullet(\CA_i) = \HOH_\bullet(X)$ then $\D^b(X) = \langle \CA_1,\dots,\CA_n \rangle$
is a semiorthogonal decomposition.
\end{corollary}
\begin{proof}
Take $\CA = \langle \CA_1,\dots,\CA_n \rangle^\perp$. Then $\CA$ is admissible and
$\D^b(X) = \langle \CA,\CA_1,\dots,\CA_n \rangle$ is a semiorthogonal decomposition.
Further, since the Hochschild homology of $X$ is the direct sum of the Hochschild
homology of the components $\CA_1$, \dots, $\CA_n$, we conclude that $\HOH(\CA) = 0$,
hence $\CA = 0$ by the Conjecture.
\end{proof}

\begin{corollary}
Let $X$ be an algebraic variety such that all integer cohomology classes are algebraic.
Let $n = \dim_\QQ(H^\bullet(X,\QQ))$. Assume that $E_1,\dots,E_n$ is an exceptional
collection in $\D^b(X)$. Then it is full, so that $\D^b(X) = \langle E_1,\dots,E_n \rangle$
is a semiorthogonal decomposition.
\end{corollary}
\begin{proof}
Note that $\HOH_\bullet(X) \cong H^\bullet(X,\kk) \cong H^\bullet(X,\QQ)\otimes_\QQ \kk \cong \kk^n$
by the HKR isomorphism.
On the other hand, $\oplus_{i=1}^n \HOH_\bullet(\lan E_i\ran) = \oplus_{i=1}^n \kk = \kk^n$.
Hence the assumptions of the previous Corollary are satisfied, so we conclude that
the collection is full.
\end{proof}

Another important consequence of the Nonvanishing Conjecture is the following

\begin{corollary}
Any increasing sequence $\CA_1 \subset \CA_2 \subset \dots $ of admissible subcategories of $\D^b(X)$ stabilizes.
\end{corollary}
\begin{proof}
From the sequence of subcategories we obtain an increasing sequence
$\HOH_\bullet(\CA_1) \subset \HOH_\bullet(\CA_2) \subset \dots$
of vector subspaces in the Hochschild homology $\HOH_\bullet(X)$. Since $\HOH_\bullet(X)$
is finite dimensional, this sequence stabilizes, hence
$\HOH_\bullet(\CA_{i}) = \HOH_\bullet(\CA_{i+1})$ for $i \gg 0$.
On the other hand, we have a semiorthogonal decomposition $\CA_{i+1} = \langle \CA_i,\fa_i \rangle$,
where $\fa_i = {}^\perp\CA_i \cap \CA_{i+1}$ is an admissible subcategory in $\D^b(X)$.
It follows that $\HOH_\bullet(\fa_i) = 0$ for $i \gg 0$, hence $\fa_i = 0$ by
the Nonvanishing Conjecture. Thus $\CA_{i+1} = \CA_i$ for $i \gg 0$.
\end{proof}

\end{document}